\documentclass{amsart}

\input xy
\xyoption{all}

\usepackage{geometry}
\usepackage{amsmath}
\usepackage{amssymb}
\usepackage{cite}
\usepackage{amsthm}  
\usepackage{tikz}
\usetikzlibrary{er,positioning}

\newcommand{\C}{\operatorname{\mathfrak{C}}}
\newcommand{\aut}[1]{\operatorname{Aut}_{#1}(\mathfrak{C})}
\newtheorem{same}{This should never appear}[section]
\newtheorem{defin}[same]{Definition}

\newtheorem{remark}[same]{Remark}
\newtheorem{theorem}[same]{Theorem}
\newtheorem{example}[same]{Example}
\newtheorem{lemma}[same]{Lemma}
\newtheorem{fact}[same]{Fact}
\newtheorem{question}[same]{Question}
\newtheorem{cor}[same]{Corollary}
\newtheorem{prop}[same]{Proposition}

\newtheorem{hypothesis}[same]{Hypothesis}
\newtheorem{nota}[same]{Notation}

\newtheorem{defin*}{Definition}
\newtheorem*{theorem*}{Theorem}

\newbox\noforkbox \newdimen\forklinewidth
\setbox0\hbox{$\textstyle\smile$}
\setbox1\hbox to \wd0{\hfil\vrule width \forklinewidth depth-2pt
 height 10pt \hfil}
\setbox\noforkbox\hbox{\lower 2pt\box1\lower 2pt\box0\relax}
\def\unionstick{\mathop{\copy\noforkbox}\limits}

\def\nonfork_#1{\unionstick_{\textstyle #1}}

\setbox0\hbox{$\textstyle\smile$}
\setbox1\hbox to \wd0{\hfil{\sl /\/}\hfil}
\setbox2\hbox to \wd0{\hfil\vrule height 10pt depth -2pt width
              \forklinewidth\hfil}
\newbox\doesforkbox
\setbox\doesforkbox\hbox{\lower 2pt\box1 \lower 2pt\box2\lower2pt\box0\relax}
\def\nunionstick{\mathop{\copy\doesforkbox}\limits}

\def\fork_#1{\nunionstick_{\textstyle #1}}

\makeatletter
\newcommand{\skipitems}[1]{%
  \addtocounter{\@enumctr}{#1}%
}

\newcommand{\ba}{\bold{a}}
\newcommand{\bb}{\bold{b}}

\newcommand{\rest}{\upharpoonright}
\newcommand{\id}{\textrm{id}}

\newcommand{\ce}{\operatorname{\mathfrak{C}}}

\newcommand{\dnf}{\unionstick}
\newcommand{\dnfb}{\overline{\dnf}}
\newcommand{\nes}{\overline{\dnf^{(nes)}}}

\newcommand{\K}{\mathbf{K}}
\newcommand{\LS}{\operatorname{LS}}

\newcommand{\leap}[1]{\le_{#1}}

\newcommand{\geap}[1]{\ge_{#1}}

\newcommand{\lea}{\leap{\K}}

\newcommand{\gea}{\geap{\K}}

\newcommand{\gtp}{\mathbf{tp}}
\newcommand{\gS}{\mathbf{S}}

\title{Simple-like independence relations in abstract elementary classes}
\date{\today.} 

\author{Rami Grossberg}
\email[Rami Grossberg]{rami@cmu.edu}
\urladdr{http://www.math.cmu.edu/$\sim$rami}
\address{Department of Mathematical Sciences\\
Carnegie Mellon University\\
Pittsburgh PA 15213, USA}

\author{Marcos Mazari-Armida}
\email{mmazaria@andrew.cmu.edu}
\urladdr{http://www.math.cmu.edu/~mmazaria/ }
\address{Department of Mathematical Sciences \\ Carnegie Mellon University \\ Pittsburgh PA 15213, USA}

\begin{document}

\maketitle

{\let\thefootnote\relax\footnote{{AMS 2010 Subject Classification: Primary 03C48. Secondary: 03C45, 03C55.
}}}  

\begin{abstract}
We introduce and study simple and supersimple independence relations in the context of AECs with a monster model.

\begin{theorem} Let $\K$ be an AEC with a monster model. 
\begin{itemize}
\item If $\K$ has a simple independence relation, then $\K$ does not have the $2$-tree property. 
\item If $\K$ has a simple independence relation with the $(<\aleph_0)$-witness property for singletons, then $\K$ does not have the tree property.
\end{itemize}
\end{theorem}

The proof of both facts is done by finding cardinal bounds to classes of small  Galois-types over a fixed model that are inconsistent for large subsets. We think that this finer way of counting types is an interesting notion in itself.

We characterize supersimple independence relations by finiteness of the Lascar rank under locality assumptions on the independence relation.

\end{abstract}

\tableofcontents

\section{Introduction}

\emph{Simple theories} were discovered by Shelah in the mid seventies, an early characterization from his 1978 book \cite{shbook} is  Theorem III.7.7.  Originally they were named \emph{theories without the tree property},   Shelah's first paper on them was published in 1980 \cite{sh93}. Simple theories were ignored  for more than a decade.  In 1991 Hrushovski circulated \cite{Hr} (which was published in 2002), there he discovered that the first-order theory of an ultraproduct of finite fields while unstable is simple in the sense of Shelah and established an early version of the \emph{type-amalgamation theorem} (also known as the independence theorem). This work was extended later by Chatzidakis and Hrushovski in the mid nineties, eventually published as \cite{ChHr}. Influenced by these papers, Kim  in \cite{kim0} and with Pillay in \cite{kipi} managed to adapt the type-amalgamation theorem from the algebraic context to complete first-order theories and solved a technical difficulty Shelah had with forking. We recommend \cite{gil} for some of the basic results, history (approved by Shelah) as well as some technical simplifications and the chain condition. The subject of simple theories and more generally studying various variants of forking-like relations for unstable first-order theories got much attention in the last 20 years as witnessed by three books dedicated to the subject: \cite{wag}, \cite{cas11}, and \cite{kim}.

In 1976 and 1977 Shelah circulated preprints of \cite{sh87a}, \cite{sh87b} and \cite{sh88} starting the far reaching program of extending his classification theory of first-order theories to several non-elementary classes. First classes axiomatizable by a theory in $L_{\omega_1,\omega}(\mathbf Q)$ and later to the more general syntax-free context of \emph{Abstract Elementary Classes (AECs for short)}. An elementary introduction to the theory of AECs can be found in \cite{grossberg2002}. A more in depth introduction is the two volume book by Shelah \cite{shelahaecbook}. Another  book is Baldwin's \cite{baldwinbook09}. For many years Shelah was the only person who managed to make progress in the field.  Much of the early work was motivated by Shelah's categoricity conjecture (a generalization of Morley's categoricity theorem). Naturally the work was closely related to generalizing first-order $\aleph_0$-stability and superstability.

There is a very extensive literature about attempts to develop analogues to $\aleph_0$-stability, superstability and stability for various classes of AECs.  Always under some extra assumptions on the AEC. This massive effort occupies  thousands of pages and is impossible to summarize in this paper. A start can be found in the above mentioned books by Baldwin and Shelah, however in the last decade much was added. See in particular in the PhD theses of Boney \cite{bonPhD} and Vasey \cite{vasPhD}. 

The goal of this paper is to begin exploring analogues of simplicity in the context of AECs. A-priori it is unclear that there is a natural property (for AECs) that correspond directly to simplicity.  It is plausible that there are several such properties. We introduce \emph{simple} and \emph{supersimple} independence relations. The main difference between stable independence relations and the relations that we introduce is that we do not assume uniqueness of non-forking extensions and instead assume the type-amalgamation property. Although this may seem like a minor change, based on our knowledge of forking in first-order theories this is actually a significant one.

 Simplicity in first-order theories can be approached from several points of view: using ranks, tree-property, axiomatic properties of forking (or independence properties in general), and counting families of types.  In this paper we too approach simplicity-like properties of AECs from various different directions.

We introduce the function $NT(\mu, \lambda, \kappa)$ to connect the existence of a simple-like independence relation with structural properties of the AEC. Our function generalizes $NT(\mu, \lambda)$ of \cite{cas}. The function $NT(\mu, \lambda, \kappa)$ assigns to each $\mu \leq \lambda$ and $\kappa$  cardinals the supremum of $|\Gamma|$ such that $\Gamma$ is a subset of Galois-types over models of size less than $\mu$ which are contained in a fixed model of size $\lambda$ and such that any subset of $\Gamma$ of cardinality greater than $\kappa$ is inconsistent. Intuitively this function let us count types in a finer way than just calculating the number of types over a fixed model.

We find the following bounds for the different kinds of independence relations studied in this paper.

\textbf{Theorem. } Let $\K$ be an AEC with a monster model.
\begin{enumerate}
\item (Theorem \ref{tp-st}) If $\dnfb$ is a stable independence relation, then \[ NT(\mu, \lambda, \kappa) \leq \lambda^{\kappa_1(\dnfb)} + \kappa^{-}. \]
\item (Theorem \ref{bsimple}) If $\dnfb$ is a simple independence relation, $\kappa(\dnfb)  \leq \mu \leq \lambda$ and $\mu^{<\ell(\dnfb)}=\mu$, then \[ NT(\mu, \lambda, \aleph_0) \leq \lambda^{\kappa(\dnfb) } + 2^{\mu}. \]
\item (Theorem \ref{bound_cond2}, \ref{supersim})
If $\dnfb$ is a simple independence relation with the $(<\aleph_0)$-witness property for singletons or a supersimple independence relation, $\kappa(\dnfb)  \leq \mu \leq \lambda$ and $\mu^{<\ell(\dnfb)}=\mu$, then \[ NT(\mu, \lambda, (2^\mu)^+) \leq \lambda^{\kappa(\dnfb)} + 2^{\mu}. \]
\end{enumerate}

We show that these bounds are useful as they imply that the AEC is stable or the failure of the tree property. The extension of the tree property to AECs is another of the contributions of the paper and is the based on the the idea that small types play the role of formulas (see Definition \ref{tree}).

\textbf{Corollary.}  Let $\K$ be an AEC with a monster model.
\begin{enumerate}
\item (Corollaries \ref{st-st}, \ref{ntp})  If $\dnfb$ is a stable independence relation independence relation, then $\K$ is stable and does not have the tree property.
\item (Corollary \ref{n2-tp}) If $\dnfb$ is a simple independence relation, then $\K$ does not have the $2$-tree property.
\item (Corollaries \ref{bound_cond3}, \ref{supersim}) If $\dnfb$ is a simple independence relation with the $(<\aleph_0)$-witness property for singletons or a supersimple independence relation, then $\K$ does not have the tree property.

\end{enumerate}

Moreover, using similar ideas to those used to prove the previous corollary, we obtain a new characterization of stable first-order theories assuming simplicity. We show that if first-order non-forking is contained in non-splitting and $T$ is simple, then $T$ is stable (Lemma \ref{star-equiv}).

In a different direction, we characterize supersimple independence relations via the Lascar rank (extended to AECs in \cite{bogr}) under the $(<\aleph_0)$-witness property for singletons. This extends \cite[2.5.16]{kim} to the AEC context.

 \textbf{Theorem \ref{equivsup}.} Assume $\K$ has a monster model. Let $\dnfb$ be a simple independence relation with the $(<\aleph_0)$-witness property for singletons. The following are equivalent.
\begin{enumerate}
\item $\dnfb$ is a supersimple independence relation.
\item If $M \in \K$ and $p\in \gS(M)$, then $U(p) < \infty$.
\end{enumerate}

A natural question whenever encountering work in pure model theory is about applications.  In this paper we do not deal with applications, we believe that it is premature to focus in applications as even for first-order simple theories the first significant applications were found more than 15 years after the basic results were discovered. Only recently some early applications were discovered of the much better understood theory of stable and superstable AECs.  For this we refer the interested reader to recent results of the second author on classes of modules, among them:  \cite{KuMaz}, \cite{m2}, \cite{m4}, \cite{m3}, and \cite{maz2}.

It is worth mentioning that there have been some efforts to extend the notion of simplicity to non-elementary settings.  Buechler and Lessman introduced a notion of simplicity for a strongly homogeneous structure in \cite{bule}, Ben-Yaccov introduced a notion of simplicity for compact abstract theories in \cite{ben},  Hyttinen and Kes\"{a}l\"{a} introduced a notion of simplicity for $\aleph_0$-stable finitary AECs with disjoint amalgamation and a prime model in \cite{hk}, and Shelah and Vasey introduced a notion of supersimplicity for $\aleph_0$-nicely stable AECs in \cite{shvas}.  One major difference between our context and that of  \cite{bule}  is that in their context types can be identified with sets of first-order formulas. As for \cite{ben}, types in his setting have a strong finitary character built in. While in our context types are orbits of the monster model $\C$ under the action of $\aut A$. As for \cite{hk} and \cite{shvas}, a major difference is that we do not assume any trace of stability.

On March 3rd, 2020, two days before posting this paper in the arXiv,  Kamsma paper \cite{kam} was posted in the arXiv. In it, he  introduced simple independence relations in AECats. Our papers study different aspects of simplicity in non-elementary classes. An important difference is that simple independence relations in his sense have finite character (called \emph{union} in his paper), while in ours they do not have it. Kamsma answers partially Question \ref{q-cano} of this paper (see Remark \ref{kams}).

The paper is organized as follows. Section 2 presents necessary background. Section 3 introduces the function $NT(\text{-}, \text{-}, \text{-})$, which is the main technical device of the paper, and a tree property. Section 4 deals with stable independence relations, a bound for $NT(\mu, \lambda, \kappa)$ is found, and it is shown that it implies stability and the failure of the tree property. We also study the consequences of weakening the uniqueness property by inclusion of the relation in explicitly non-splitting. Section 5 introduces simple independence relations, a bound for $NT(\mu, \lambda, \aleph_0)$ is found and it is shown that it implies the failure of the $2$-tree property. Section 6 studies simple independence relations with locality assumptions. A bound for $NT(\mu, \lambda, (2^\mu)^+)$ is found and it is shown that it implies the failure of the tree property. Section 7 introduces supersimple independence relations and characterizes them by the Lascar rank. It is also shown that the existence of a supersimple independence relation in a class that admits intersections implies the $(<\aleph_0)$-witness property for singletons.

This paper was written while the second author was working on a Ph.D. under the direction of the first author at Carnegie Mellon University and the second author would like to thank the first author for his guidance and assistance in his research in general and in this work in particular. We thank Hanif Cheung for helpful conversations. We would also like to thank Mark Kamsma, Samson Leung, Sebastien Vasey, and a couple of referees for comments that helped improve the paper.

\section{Preliminaries}
We assume the reader has some familiarity with abstract elementary classes as presented for example in \cite[\S 4 - 8]{baldwinbook09} and  \cite[\S 2, \S 4.4]{ramibook}. Familiarity with \cite{lrv} would be useful, but it is not required as we will recall the notions from \cite{lrv} that are used in this paper. We begin by quickly introducing the basic notions of AECs that we will use in this paper.

 Since the main results of the paper assume joint embedding, amalgamation and no maximal models, we will assume these since the beginning.\footnote{Some of the definitions presented here make sense without these hypothesis.}

\begin{hypothesis}\label{hyp}
Let $\K$ be an AEC with joint embedding, amalgamation and no maximal models.
\end{hypothesis}

\subsection{Basic concepts} We begin by introducing some model theoretic notation.

\begin{nota}\
\begin{itemize}
\item If $M \in \K$, $|M|$ is the underlying set of $M$ and $\| M \|$ is the cardinality of $M$.
\item If $\lambda$ is a cardinal, $\K_{\lambda} =\{ M \in \K : \| M \| =\lambda \}$ and $\K_{< \lambda} = \{ M \in \K : \| M \| <\lambda \}$.
\item If $M \in \K$ and $\lambda \leq \| M \|$, $[M]^\lambda=\{ N : N \lea M \} \cap \K_\lambda$ and $[M]^{< \lambda}=\{ N : N \lea M \} \cap \K_{< \lambda}$.
\item Let $M, N \in \K$. If we write ``$f: M \to N$" we assume that $f$ is a $\K$-embedding, i.e., $f: M \cong f[M]$ and $f[M] \lea N$. 

\end{itemize}
\end{nota}

We will also use the next set theoretic notation.
\begin{nota}\
\begin{itemize}
\item For $\kappa$ a cardinal, we define $\kappa^{-}=\theta$ if $\kappa=\theta^+$ and $\kappa^{-}=\kappa$ otherwise.
\item For $\kappa$ a cardinal and $\kappa \leq |A|$, let $\mathcal{P}_{< \kappa}(A)=\{ B \subseteq A : | B | < \kappa \}$.
\end{itemize}
\end{nota}

Recall the following definitions due to Shelah.

\begin{defin} Let $M \in \K$.
\begin{enumerate}
\item$M$ is $\lambda$-universal if for every $N \in \K_{< \lambda}$, there exists $f: N \to M$.
\item $M$ is $\lambda$-model homogeneous if for every $M_0 \lea N_0$ both in $\K_{< \lambda}$, if $M_0 \lea M$ then there exists $f: N_0 \xrightarrow[M_0]{} M$.
\end{enumerate}

\end{defin}

\begin{remark}
Since $\K$ has joint embedding, amalgamation and no maximal models, we work inside a monster model $\ce$ (as in complete first-order theories). A monster model $\ce$ is large compared to all the models we consider and is universal and model homogeneous for small cardinals. As usual, we assume that all the elements and sets we consider are contained in the monster model $\ce$. Further details are given in \cite[\S 7]{notesv}.
\end{remark}

Shelah introduced a notion of semantic type in \cite{sh300}. The
original definition was refined and extended by many authors who
following \cite{grossberg2002} call these semantic types Galois-types (Shelah recently named them orbital
types). We present here the modern definition and call them Galois-types
throughout the text.  We use the terminology of \cite[2.5]{mv} and introduce Galois-types without using the monster model.

\begin{defin}\label{gtp-def} \
  \begin{enumerate}
    \item Let $\K^3$ be the set of triples of the form $(\bb, A, N)$, where $N \in \K$, $A \subseteq |N|$, and $\bb$ is a sequence of elements from $N$. 
    \item For $(\bb_1, A_1, N_1), (\bb_2, A_2, N_2) \in \K^3$, we say $(\bb_1, A_1, N_1)E (\bb_2, A_2, N_2)$ if $A := A_1 = A_2$, and there exists $f_\ell : N_\ell \xrightarrow[A]{} N$ such that $f_1 (\bb_1) = f_2 (\bb_2)$.
    \item Note that $E$ is an equivalence relation on $\K^3$. It is transitive because $\K$ has amalgamation.
    \item For $(\bb, A, N) \in \K^3$, let $\gtp_{\K} (\bb / A; N) := [(\bb, A, N)]_E$. We call such an equivalence class a \emph{Galois-type}. If $N = \ce$ (where $\ce$ is a monster model) we write $\gtp(\ba /A)$ instead of $\gtp(\ba/A ; \ce)$.
\item For $N \in \K$, $A \subseteq N$ and $I$ a non-empty set, $\gS^I(A ;N)= \{ \gtp(\bb/ A; N) : \bb =  \langle b_i \in N :  i \in I \rangle \}$. Let $\gS(M):= \gS^1(M)$ and $\gS^{<\infty}(M):= \bigcup_{\alpha < \infty} \gS^\alpha(M)$. 
\item An AEC is \emph{$\lambda$-Galois-stable} if for any $M \in \K_\lambda$ it holds that $| \gS(M) | \leq \lambda$. An AEC is \emph{stable} if there is $\lambda \geq \LS(\K)$ such that $\K$ is $\lambda$-Galois-stable.
\item For $p= \gtp_{\K} ((b_i)_{i \in I} / A; N) \in \gS^I(A; N) $, $A' \subseteq A$ and $I_0 \subseteq I$, $p^{I_0}\rest_{A'}:= [((b_i)_{i\in I_0}, A', N)]_E$.
  \end{enumerate}
\end{defin}

The following fact shows that in the presence of a monster model, the Galois-type of $\bb$ over a set $A$ is simply the orbit of $\bb$ under the action of the automorphisms of $\ce$ fixing $A$. 
\begin{fact}
 $\gtp(\bb_1 /A; \ce)= \gtp(\bb_2/A; \ce)$ if and only if there exists $f\in Aut_A(\ce)$ with $f(\bb_1)=\bb_2$.
\end{fact}

The notion of tameness was isolated by the first author and VanDieren in \cite{tamenessone} and type-shortness by Boney in \cite{bont}.

\begin{defin} \
\begin{itemize}
\item $\K$ is \emph{$(< \kappa)$-tame} for $\theta$-types if for any $M \in \K$ and $p \neq q \in \gS^I(M)$ with $| I |=\theta$,  there is $A \in \mathcal{P}_{<\kappa}( M)$ such that $p\upharpoonright_{A} \neq q\upharpoonright_{A}$. 
\item $\K$ is \emph{$\kappa$-tame} for $\theta$-types if it is \emph{$(< \kappa^+)$-tame} for $\theta$-types.
\item $\K$ is \emph{fully $(<\kappa)$-tame} if for every $\theta$ ordinal, $\K$ is $(< \kappa)$-tame for $\theta$-types.
\item $\K$ is \emph{fully $(<\kappa)$-tame and -type-short} if for any $M \in \K$ and $p \neq q \in \gS^I(M)$, there is $A\in \mathcal{P}_{<\kappa}( M)$ and $I_0 \in \mathcal{P}_{<\kappa}( I)$ such that  $p^{I_0}\rest_A \neq q^{I_0}\rest_A$.
\end{itemize}
\end{defin}

\subsection{Independence relations and the witness property} Global independence relations in the context of AECs and $\mu$-AECs have been extensively studied in the last few years, see for example \cite{bogr}, \cite{vasey16}, and \cite{lrv}. Below we introduce a weak independence notion. Our notation and choice of axioms is inspired by \cite{lrv} and the specific independence relations that we will study in this paper.

\begin{defin}
$\dnfb$ is an \emph{independence relation} in an AEC $\K$ if the following properties hold:
\begin{enumerate}
\item $\dnfb \subseteq \{(M, A, B) : M\lea \ce \text{ and } A, B \subseteq \ce \}$. We say that $\gtp(\bar{a}/B)$ does not fork over $M$ if $ran(\bar{a}) \dnfb_{M}  B$. This is well-defined by the next three properties.
\item (Preservation under $\K$-embeddings) Given $M \lea \ce$, $A, B \subseteq \ce$ and $f \in Aut(\ce)$, we have that $A \dnfb_{M}  B$ if and only if $f[A] \dnfb_{f[M]} f[B]$.
\item (Monotonicity) If $A \dnfb_{M}  B$ and $A_0 \subseteq A$, $B_0 \subseteq  B$, then $A_0 \dnfb_{M}  B_0$.
\item (Normality) $A \dnfb_{M}  B$ if and only if $A\cup M \dnfb_{M}  B\cup M$.
\item (Base monotonicity) If $A \dnfb_{M}  B$, $M \lea N \lea \ce$ and $|N| \subseteq B$, then $A \dnfb_{N}  B$.
\item (Existence) If $M\lea N$ and $p \in \gS^{<\infty}(M)$, then there exists $q\in \gS^{< \infty}(N)$ extending $p$ such that $q$ does not fork over $M$.
\item (Transitivity) If $M \lea N$, $A \dnfb_{M}  N$ and $A \dnfb_{N}  B$, then $A \dnfb_{M}  B$.

\end{enumerate}
\end{defin}

Let us introduce some notation.

\begin{nota} Given $\dnfb$ an independence relation:
\begin{itemize}
\item For $\alpha$ a cardinal, let $\kappa_\alpha(\dnfb)$ be the minimum $\lambda$ (or $\infty$) such that: If $p \in \gS^{\alpha}(M)$, then there exists $M_0 \lea M$ with $|| M_0 || \leq \lambda$ and $p$ does not fork over $M_0$.   
\item  Let $(\kappa(\dnfb), \ell(\dnfb))$ be the minimum pair $(\lambda, \theta)$ of cardinals\footnote{$\lambda$ is an infinite cardinal, but $\theta$ might be a finite cardinal. The minimum is taken with respect to the canonical ordering in pairs of ordinals.} (or $(\infty, \infty)$) such that: If $p\in \gS^\alpha(M)$, there exists $M_0 \in \K$ with $M_0 \lea M$, $\| M_0 \| \leq \lambda + \alpha^{<\theta}$ and $p$ does not fork over $M_0$. 
\end{itemize}
\end{nota}





The following notion is a  locality notion for independence relations.

\begin{defin}[{\cite[3.12.(9)]{vasey16}}] \label{wp}  Let $\dnfb$ be an independence relation.
$\dnfb$ has the \emph{$(<\theta)$-witness property} of length $\alpha$ if for all $M \lea N$ and $\bb \in \ce^{\alpha}$: $\bb \dnfb_M  N$ if and only if $\bb \dnfb_M A$ for every  $ A \in \mathcal{P}_{<\theta}( N)$. We say that $\dnfb$ has the $(<\theta)$-witness property if and only if $\dnfb$ has the $(<\theta)$-witness property of length $\alpha$  for all $\alpha$. 
\end{defin}

Observe that since first-order non-forking has finite character, first-order non-forking has the $(<\aleph_0)$-witness property. This might not be the case for independence relations as the next example shows. This example was first considered in \cite[1.43]{adler}.

\begin{example}
Let $L(\K)=\emptyset$ and $\K = (Sets, \subseteqq)$. Given $M, A, B \in \K$ let:
\[ A \dnfb_M B \text{ if and only if } | (A\cap B) \backslash M | \leq \aleph_0 \]

It is easy to show that $\dnfb$ is an independence relation. $\dnfb$ has the $(<\aleph_0)$-witness property of length $\alpha$ for $\alpha$ countable, but not for $\alpha$ uncountable. Hence $\dnfb$ does not have the  $(<\aleph_0)$-witness property.
\end{example}

In a few places in the paper we will assume that the independence relation under consideration has the witness property in order to be able to carry out some of the proofs (see for example Lemma \ref{bound_cond} and Theorem \ref{equivsup}).

The next lemma gives a natural condition that implies the witness property. It fixes a small gap in \cite[4.3]{vasey16}; the argument in \cite[4.3]{vasey16} seems to only work for $M$ of cardinality less than or equal to $\kappa_\alpha(\dnfb)$ as we need $M \lea N$ in order to apply transitivity.

\begin{lemma}\label{lcwitness}
Let $\dnfb$ be an independence relation. If $\kappa_\alpha(\dnfb)=\lambda$, then $\dnfb$ has the $(<\lambda^+)$-witness property of length $\alpha$.
\end{lemma}
\begin{proof}
The proof is divided into two cases:

\underline{Case 1:} Assume that $\| M \| \leq \lambda$. Let $M \lea N$ and $\ba \in \ce^\alpha$, by $\kappa_\alpha(\dnfb)=\lambda$ there is $N' \in [N]^\lambda$ such that $\ba \dnfb_{N'} N$. Since $\| M \| \leq \lambda$ and $M \lea N$, we may assume without lost of generality that $M \lea N'$. Moreover, as $N' \in \mathcal{P}_{\leq \lambda}(N)$, we have that  $\ba \dnfb_M N'$. Then by transitivity we conclude that $\ba \dnfb_ M N$.

\underline{Case 2:} Assume that $\| M \| > \lambda$.  Let $M \lea N$ and $\ba \in \ce^\alpha$. Since $\kappa_\alpha(\dnfb)=\lambda$ there is $M' \in [M]^\lambda$ such that $\ba \dnfb_{M'} M$. Using that $\forall B \in \mathcal{P}_{\leq \lambda}(N)(\ba \dnfb_M B)$ and transitivity, it follows that $\forall B \in \mathcal{P}_{\leq \lambda}(N)(\ba \dnfb_{M'} B)$. Then by the first case we have that $\ba \dnfb_{M'} N$. Hence $\ba \dnfb_M N$ by base monotonicity. \end{proof}

We will give a few other natural conditions that imply the witness property, see for example Fact \ref{lc3} and Corollary \ref{lc4}.

\section{The basic notions}

In this section we introduce a way of counting Galois-types over small submodels and generalize the tree property to AECs. We think that this finer way of counting types is an interesting notion in itself. As mentioned in the preliminaries we are assuming Hypothesis \ref{hyp}.

In this paper Galois-types over submodels will play a central role.

\begin{defin}\label{dNT} Let $M \in \K$ and $\mu \leq \| M\|$: 
\[ \gS(M, \leq\mu) = \{ \gtp(a/N) : N \lea M \text{ and } \| N \| \leq \mu \} \]
\end{defin}

\begin{defin}
Let $\Gamma$ be a set of Galois-types. $\Gamma$ is \emph{consistent} if there is $a \in \ce$ such that $a$ realizes every Galois-type in $\Gamma$, i.e., $\gtp(a/dom(p))=p$ for each $p \in \Gamma$. If such an $a \in \ce$ does not exist we say that $\Gamma$ is \emph{inconsistent}.
\end{defin}

The following notion generalizes \cite[2.3]{cas} to the AEC setting. 

\begin{defin}
Let $\mu, \lambda \in  [\LS(\K), \infty)$ such that $\mu \leq \lambda$ and $\kappa$ a cardinal (possibly finite). We define the following:
\[ NT(\mu, \lambda, \kappa)=sup\{ |\Gamma| : \exists M \in \K_\lambda( \Gamma \subseteq \gS(M, \leq \mu) \text{ and } \forall \Delta \subseteq \Gamma( |\Delta| \geq \kappa \rightarrow \Delta \text{ is inconsistent}) )\} \] 

If $\kappa =2$  instead of writing $NT(\mu, \lambda, 2)$, we write $NT(\mu, \lambda)$ as in \cite{cas}.\footnote{The definition given here does not fully match the definition of \cite{cas} when $\K=(Mod(T), \preceq)$ for a complete first-order theory $T$, since the bound $\mu$ on \cite{cas} refers to the cardinality of the type (the number of formulas in it) while in our definition it refers to the cardinality of the domain of the type.}
\end{defin}

The following bounds are easy to calculate and hold in general. In what follows, see Theorems \ref{tp-st}, \ref{star-bound}, \ref{bsimple} and \ref{bound_cond2}, we will find sharper bounds which will be the key to show stability or the failure of the tree property under additional hypothesis.

\begin{prop}\label{first-b}\

\begin{enumerate}
\item If $M \in \K_\lambda$, then $| \gS(M) |\leq NT(\lambda, \lambda, 2)$.
\item If $\mu_1 \leq \mu_2$, $\lambda_1 \leq \lambda_2$ and $\kappa_1 \leq \kappa_2$ then $NT(\mu_1, \lambda_1, \kappa_1) \leq NT(\mu_2, \lambda_2, \kappa_2)$.
\item If $\mu \leq \lambda$, then the value of  $NT(\mu, \lambda, \text{-})$ is bounded as follows:
\begin{enumerate}
\item If $\kappa \in [2, (\lambda^\mu)^+]$, then $NT(\mu, \lambda, \kappa) \leq \lambda^\mu$.
\item If $\kappa \in ((\lambda^\mu)^+, (2^\lambda)^+]$, then $NT(\mu, \lambda, \kappa) \leq 2^\lambda$.
\item If $\kappa \in ((2^\lambda)^+, 2^{\lambda^\mu}]$, then $NT(\mu, \lambda, \kappa) \leq 2^{\lambda^\mu}$.
\end{enumerate}

\item  $\K$ is $\lambda$-Galois-stable if and only if $NT(\mu, \lambda, \kappa) \leq \lambda$ for every $\mu \in [\LS(\K), \lambda]$ and $\kappa\in [2, \lambda^+]$. 

\end{enumerate}
\end{prop}
\begin{proof}\
\begin{enumerate}
\item Let $\chi=| \gS(M)|$ and $\{ p_\alpha : \alpha < \chi \}$ an enumeration without repetitions of $\gS(M)$. Observe $ \{ p_\alpha : \alpha < \chi \} \subseteq \gS(M, \leq\lambda)$ and any set $\{p_\alpha, p_\beta\}$ is inconsistent if $\alpha \neq \beta$. Therefore, $|\gS(M) | = \chi \leq NT(\lambda, \lambda, 2)$. 

\item Follows from the fact that if  $\Gamma \subseteq \gS(M, \leq\mu_1)$ for $M \in \K_{\lambda_1}$ and each subset of size greater or equal to $\kappa_1$ is inconsistent, then there is $M^* \in \K_{\lambda_2}$ with $M \lea M^*$ and $\Gamma \subseteq \gS(M^*, \leq \mu_2)$ such that any subset of size greater or equal to $\kappa_2$ is inconsistent. 

\item
\begin{enumerate}
\item Let $\kappa \in  [2, (\lambda^\mu)^+]$, $\chi := \lambda^\mu$ and $\{p_\alpha : \alpha < \chi^+\} \subseteq \gS(M, \leq \mu)$ for $M \in \K_\lambda$. 

 Let $\Phi : \chi^+ \to [M]^{\leq \mu}$ be defined as $\Phi(\alpha) = dom(p_\alpha)$, since  $|[M]^{\leq \mu}|=\lambda^\mu$ by the pigeonhole principle there is $S \subseteq \chi^+$  of size $\chi^+$ and $N \in  [M]^{\leq \mu}$ such that $dom(p_\alpha)=N$ for each $\alpha \in S$. Let $\Psi: S \to \gS(N)$ be defined as $\Psi(\alpha)= p_\alpha$, since $|\gS(N)| \leq 2^{\mu}$ by the pigeonhole principle there is $S' \subseteq S$ of size $\chi^+$ and $q \in \gS(N)$ such that $p_\alpha = q$ for each $\alpha \in S'$. In particular $\{ p_\alpha : \alpha \in S'\}$ is a consistent set of size $\chi^+$. Hence $NT(\mu,\lambda,\kappa) \leq \lambda^\mu$.

\item Let $\kappa \in  ((\lambda^\mu)^+, (2^\lambda)^+]$, $\chi := 2^\lambda$ and $\{p_\alpha : \alpha < \chi^+\} \subseteq \gS(M, \leq \mu)$ for $M \in \K_\lambda$. 

Given $\alpha < \chi^+$, let $q_\alpha \in \gS(M)$ such that $q_\alpha \geq p_\alpha$, it exists because we assumed that $\K$ has amalgamation. Let $\Phi: \chi^+ \to \gS(M)$ be defined as $\Phi(\alpha)=q_\alpha$, since $|\gS(M)| \leq 2^\lambda$ by the pigeonhole principle there is $S \subseteq \chi^+$ of size $\chi^+$ and $q \in \gS(M)$ with $q_\alpha = q$ for every $\alpha \in S$. In particular $\{ p_\alpha : \alpha \in S'\}$ is a consistent set of size $\chi^+$. Hence $NT(\mu,\lambda,\kappa) \leq 2^\lambda$.

\item Similar to (b).
 
\end{enumerate}

\item The forward direction is similar to (3).(a) but using that for every $M\in \K_\lambda$ we have that $|\gS(M)| \leq \lambda$ instead of only $| \gS(M) | \leq 2^\lambda$. The backward direction follows from (1).



\end{enumerate}
\end{proof}

The next concept extends the tree property to the AEC context. The main idea is that Galois-types over \emph{small sets} in AECs play a similar role as that of formulas in first-order theories. This correspondence is explored in \cite{vasey16b}.

\begin{defin}\label{tree}
Let $\mu, \lambda \in [\LS(\K), \infty)$ and $k < \omega$. $\K$ has the $(\mu, \lambda, k)$-tree property if there is $\{ (a_\eta, B_\eta) : \eta \in  {}^{< \mu}\lambda \}$\footnote{As always we assume that $\forall \eta( a_\eta \in \ce \text{ and } B_\eta \subseteq \ce)$.} such that:
\begin{enumerate}
\item $\forall  \eta \in  {}^{< \mu}\lambda ( |B_\eta | < \LS(\K))$.
\item $\forall \nu \in {}^\mu\lambda ( \{ \gtp(a_{\nu \rest_{\alpha}}/ B_{\nu \rest_{\alpha}}) : \alpha < \mu \} \text{\, is consistent} )$.
\item $\forall  \eta \in  {}^{< \mu}\lambda( \{\gtp(a_{\eta^{\wedge}\alpha}/ B_{\eta^{\wedge}\alpha}) : \alpha < \lambda \} \text{ is $k$-contradictory})$.
\end{enumerate}

We say that $\K$ has the $k$-tree property if for all $\mu, \lambda \in [\LS(\K), \infty)$ $\K$ has the $(\mu, \lambda, k)$-tree property and $\K$ has the tree property if there is a $k < \omega $ such that $\K$ has the $k$-tree property.

\end{defin}

The following lemma relates the two concepts we just introduced. A similar construction in the first-order context appears in\cite[2.3]{cas}.

\begin{lemma}\label{2-tp}
Assume $\lambda^{<\mu} = \lambda$ and $\LS(\K) \leq \mu \leq \lambda$. If $\K$ has the $(\mu, \lambda, 2)$-tree property, then $NT(\mu, \lambda,2) = \lambda^{\mu}$. Moreover,  $NT(\mu, \lambda, \kappa) \geq \lambda^\mu$ for all $\kappa \geq 2$.\footnote{As usual we assume that $\lambda, \mu$ are cardinals way below the size of the monster model.}
\end{lemma}
\begin{proof}
By the definition of the tree property we have  $\{ (a_\eta, B_\eta) : \eta \in  {}^{< \mu}\lambda \}$ such that:
\begin{enumerate}
\item $\forall  \eta \in  {}^{< \mu}\lambda ( |B_\eta | < \LS(\K))$.
\item $\forall \nu \in {}^\mu\lambda ( \{ \gtp(a_{\nu \rest_{\alpha}}/ B_{\nu \rest_{\alpha}}) : \alpha < \mu \} \text{\, is consistent})$.
\item $\forall  \eta \in  {}^{< \mu}\lambda( \{\gtp(a_{\eta^{\wedge}\alpha}/ B_{\eta^{\wedge}\alpha}) : \alpha < \lambda \} \text{ is $2$-contradictory})$.
\end{enumerate}

Let $A=\bigcup_{\eta \in  {}^{< \mu}\lambda} B_{\eta}$. Since $\lambda^{<\mu} = \lambda$ and each $B_\eta$ has cardinality less than $\LS(\K)$, we have that $|A|\leq \lambda$. So applying downward L\"{o}wenheim-Skolem in $\ce$ we obtain $M \in \K_\lambda$ such that $\forall  \eta \in  {}^{< \mu}\lambda ( B_\eta \subseteq |M|)$.

For each  $\nu \in {}^\mu\lambda$, pick $a_\nu \in \ce$ realizing $\{ \gtp(a_{\nu \rest_{\alpha}}/ B_{\nu \rest_{\alpha}}) : \alpha < \mu \}$ and apply downward L\"{o}wenheim-Skolem to  $\bigcup_{\alpha < \mu} B_{\nu \rest_{\alpha}}$ in $M$ to get $M_\nu \in [M]^{\leq \mu}$. Then define $p_\nu := \gtp(a_\nu/ M_\nu)$.

Observe that $\{ p_\nu : \nu\in  {}^\mu\lambda\} \subseteq \gS(M, \leq\mu)$ and using part (3) of the definition of the tree property it is easy to show that: if $\nu_1 \neq \nu_2$, then $p_{\nu_1} \neq p_{\nu_2}$. Therefore $|\{ p_\nu  : \nu\in  {}^\mu\lambda\}|= \lambda^{\mu}$. Moreover, using  part (3) of the definition of the tree property it follows that any pair of types is inconsistent. Hence $NT(\mu, \lambda, 2) \geq \lambda^{\mu}$.

The equality and moreover part follow from Proposition \ref{first-b}.
\end{proof}

As we will see later, if  we only know that $\K$ has the tree property it becomes more complicated to obtain a lower bound on $NT(\text{-},\text{-},\text{-})$.

\section{Stable independence relations}

In this section we deal with stable independence relations. The definition given here for a stable independence relation is similar to the one given in \cite{lrv}. The properties given here are obtained by taking the  ``closure" of a stable independence relation in the sense of \cite{lrv}; this is formalized in \cite[8.2]{lrv}. An important difference with \cite{lrv} is that we do not assume the witness property.

\begin{defin}[{\cite[8.4, 8.5, 8.6]{lrv}}]\label{sta-ind}
 $\dnfb$ is a \emph{stable independence relation} in $\K$ if the following properties hold:
\begin{enumerate}
\item $\dnfb$ is an independence relation. 
\item (Symmetry)  $A \dnfb_{M}  B$ if and only if $B \dnfb_{M}  A$.
\item (Uniqueness) Let $p,q \in \gS^{<\infty}(B; N)$ with $M \lea N$ and $| M | \subseteq B \subseteq |N| $. If $p\rest_{M}= q\rest_M$ and $p,q$ do not fork over $M$, then $p=q$.  
\item (Local character) For each cardinal $\alpha$ there exists a cardinal $\lambda$  (depending on $\alpha$) such that: If $p \in \gS^{\alpha}(M)$, then there exists $M_0 \lea M$ with $|| M_0 || \leq \lambda$ and $p$ does not fork over $M_0$.  
\end{enumerate}
\end{defin}

We begin by bounding $NT(\text{-},\text{-},\text{-})$.

\begin{theorem}\label{tp-st}
If $\dnfb$ is a stable independence relation, then \[ NT(\mu, \lambda, \kappa) \leq \lambda^{\kappa_1(\dnfb)} + \kappa^{-}. \]
In particular, we get that $NT(\mu, \lambda) \leq \lambda^{\kappa_1(\dnfb)}$.
\end{theorem}
\begin{proof}
Let $\lambda_0=\kappa_1(\dnfb)$, $\chi = \lambda^{\lambda_0} + \kappa^{-}$ and $\{ p_\alpha :\alpha < \chi^+ \} \subseteq \gS(M, \leq \mu)$ for $M \in \K_\lambda$. 

By local character for every $\alpha < \chi^+$ there is  $R_\alpha \in [M]^{\lambda_0}$ such that $p_\alpha \text{ does not fork over } R_\alpha$. We define $\Phi: \chi^+ \to [M]^{\lambda_0}$ as $\Phi(\alpha)=R_\alpha$. Then by the pigeonhole principle  there is $R \in [M]^{\lambda_0}$ and $S \subseteq \chi^+$ of cardinality $\chi^+$ such that  $p_\alpha \text{ does not fork over }  R$ for every $\alpha \in S$. Now define $\Psi: S \to  \gS(R)$ as $\Psi(\alpha)= p_\alpha\rest_{R}$, since $|\gS(R)| \leq 2^{\lambda_0}$, by the pigeonhole principle  there is $p \in \gS(R)$ and $S'\subseteq S$ of size $\chi^+$ such that $p_\alpha\rest_{R}=p$ for every $\alpha \in S'$. Observe that $p_\alpha \geq p \text{ and } p_\alpha \text{ does not fork over }  R$ for every $\alpha \in S'$.

By the extension property and transitivity for each $\alpha \in S'$, there is $q_\alpha \in \gS(M) $ extending $p_\alpha$ such that $q_\alpha\text{ does not fork over } R$. Then by uniqueness, using that for all $\alpha, \beta \in S'$ we have that $q_\alpha\rest_{R}=p_\alpha\rest_{R}=p=p_\beta\rest_{R}=q_\beta\rest_{R}$ and that both $q_\alpha, q_\beta\text{ do not fork over } R$, we conclude that there is $q \in \gS(M)$ such that $q_\alpha=q$ for every $\alpha \in S'$.  In particular, $\{p_\alpha : \alpha \in S'\}$ is consistent and $|S'|\geq \kappa$. Hence $NT(\mu, \lambda, \kappa) \leq  \lambda^{\lambda_0} + \kappa^-$.

\end{proof}

The next corollary follows directly from Proposition \ref{first-b} and the above theorem. A version of it already appears in \cite[5.17]{bgkv} and \cite[8.15]{lrv}.

\begin{cor}\label{st-st}
If $\dnfb$ is a stable independence relation, then $\K$ is $\lambda$-Galois-stable for every $\lambda$ such that $\lambda^{\kappa_1(\dnfb)}=\lambda$. 
\end{cor}

We show that the existence of a stable independence relation implies the failure of the tree property.

\begin{lemma}\label{ntp}
If $\K$ has $\dnfb$ a stable independence relation, then $\K$ does not have the tree property.
\end{lemma}
\begin{proof}
Let $\kappa_1(\dnfb)=\lambda_0$ and $k < \omega$ such that $\K$ has the $k$-tree property. Let $\mu =\lambda_0 ^+$ and $\lambda=\beth_{\mu}(\mu)$. By the definition of the $(\mu, \lambda, k)$-tree property there are $\{ (a_\eta, B_\eta) : \eta \in  {}^{< \mu}\lambda \}$ such that:
\begin{enumerate}
\item $\forall  \eta \in  {}^{< \mu}\lambda ( \|B_\eta \| < \LS(\K))$.
\item $\forall \nu \in {}^\mu\lambda ( \{ \gtp(a_{\nu \rest_{\alpha}}/ B_{\nu \rest_{\alpha}}) : \alpha < \mu \} \text{\, is consistent } )$.
\item $\forall  \eta \in  {}^{< \mu}\lambda( \{\gtp(a_{\eta^{\wedge}\alpha}/ B_{\eta^{\wedge}\alpha}) : \alpha < \lambda \} \text{ is $k$-contradictory})$.
\end{enumerate}

Realize that $\lambda^{<\mu}=\lambda$, so doing a similar construction to that of Lemma \ref{2-tp} we have  $M \in \K_\lambda$ and for each $\nu \in \lambda^{\mu}$  we fix $p_\nu=\gtp(a_\nu/ M_\nu)$ such that $M_\nu \in [M]^{\leq \mu}$ and $\forall \alpha < \mu( \gtp(a_{\nu \rest_{\alpha}}/ B_{\nu \rest_{\alpha}}) \leq p_\nu)$. 

Observe that if $A \subseteq {}^\mu\lambda$ and $\{ p_\nu : \nu \in A \}$ is consistent then the tree $\{\nu\rest_{\alpha} : \alpha < \mu, \nu \in A \}$ is finitely branching by condition (3) of the tree property, hence $|A| \leq 2^\mu$. Therefore we can conclude that for all $\Delta \subseteq \{ p_\nu : \nu \in \lambda^\mu \}$, if $|\Delta| \geq (2^\mu)^+$, then $\Delta$ is inconsistent.

Since $cf(\lambda)=\mu$, by K\"{o}nig Lemma, we have that $\lambda^\mu=  \beth_{\mu}(\mu)^\mu \geq   \beth_{\mu}(\mu)^+= \lambda^+$. We claim that $|\{ p_\nu : \nu \in \lambda^\mu \}| \geq \lambda^+$. If it was not the case, then there would be $S \subseteq \lambda^{\mu}$ with $|S|=\lambda^+$ and $\{ p_\mu : \nu \in S\}$ consistent; but this would contradict the previous paragraph since  $(2^\mu)^+ <   \beth_{\mu}(\mu)^+=\lambda^+$.  Hence \begin{equation}\lambda^+\leq NT(\mu, \lambda, (2^{\mu})^+). 
\end{equation}

On the other hand, by Theorem \ref{tp-st}, we have that  $NT(\mu, \lambda, (2^{\mu})^+)\leq \lambda^{\lambda_0} + 2^{\mu}$. Moreover, one can show that $\lambda^{\lambda_0} = \lambda$ and that $2^{\mu} \leq \lambda$, hence \
\begin{equation}
NT(\mu, \lambda, (2^{\mu})^+) \leq \lambda.\end{equation}

The last two equations give us a contradiction. \end{proof}

The above proof can also be carried out in Shelah's context of good frames, see \cite[\S II]{shelahaecbook} or  \cite[\S3]{m1} for the definition.
\begin{cor}
Let $\K$ be an AEC. If $\K$ has a type-full good $[\lambda_0, \infty)$-frame, then $\K$ does not have the tree property.
\end{cor}
\begin{proof}[Proof sketch] 
Using local character (in the sense of a good frame) it is easy to show by induction on $\|M\|$ that for every  $p \in \gS(M)$ there is  $N \in [M]^{\lambda_0}$ such that $p \text{ does not fork over } N$. Using this fact together with the properties of  type-full good $[\lambda_0, \infty)$-frame one can show that the proofs of Theorem \ref{tp-st} and Lemma \ref{ntp} go through.
\end{proof}

\begin{remark}
The above corollary goes through in the weaker setting of a type-full $good^{-}[\lambda_0, \infty)$-frame (see \cite[3.5.(4)]{m1}).
We do not know if it still goes through in the even weaker setting of w-good frames (see \cite[3.7]{m1}).
\end{remark}

\subsection{Almost-stable independence relations} In this small subsection, we study what happens if instead of assuming uniqueness of extensions one assumes that the independence relation is contained in non-splitting. We show that this weaker assumption still implies stability of the AEC and the existence of a sub$\mu$-AEC with a stable independence relation. Moreover, the results in this subsection are used to obtain a new characterization of stable first-order theories assuming simplicity.  A similar notion is studied in \cite[\S 6]{shvas} under stability assumptions.

A generalization of non-splitting to AECs was introduced in \cite{bgkv}.

\begin{defin}[{\cite[3.14]{bgkv}}] We say that $A$ \emph{does not explicitly split} from $B$ over $M$, denoted by $A \nes_M B$,
if and only if for every $B_1, B_2 \subseteq B$, if $\gtp(B_1/M)=\gtp(B_2/M)$ then $\gtp(AB_1/M)=\gtp(AB_2/M)$.\end{defin}

To ease the reference to stable independence relations without uniqueness but contained in explicitly non-splitting, we introduce the following notion.

\begin{defin}\label{star-sim}
$\dnfb$ is an \emph{almost-stable independence relation} in $\K$ if the following hold: 

\begin{enumerate}
\item $\dnfb$ is an independence relation. 
\item (Symmetry)  $A \dnfb_{M}  B$ if and only if $B \dnfb_{M}  A$.
\item (Local character) For each cardinal $\alpha$ there exists a cardinal $\lambda$  (depending on $\alpha$) such that: If $p \in \gS^{\alpha}(M)$, then there exists $M_0 \lea M$ with $|| M_0 || \leq \lambda$ and $p$ does not fork over $M_0$.  Recall that $\kappa_\alpha(\dnfb)$ is the least $\lambda$ given a fixed cardinal $\alpha$. 
\item $\dnfb \subseteq \nes$.
\end{enumerate}
\end{defin}

\begin{remark}
It follows from \cite[4.2]{bgkv} that if $\dnfb$ is a stable independence relation, then $\dnfb \subseteq \nes$. Hence, a stable independence relation is an almost-independence relation.
\end{remark}

We begin by showing that a class with  an almost-stable independence relation is tame. This extends \cite[8.16]{lrv} as they prove it for stable independence relations.

\begin{lemma}\label{star-tame}
If $\dnfb$ is an almost-stable independence relation, then $\K$ is $\kappa_{2\alpha}(\dnfb)$-tame for types of length $\alpha$.
\end{lemma}
\begin{proof}
Let $N \in \K$ and $p, q \in \gS^{\alpha}(N)$ such that $p\rest_D = q\rest_D$ for every $D \in \mathcal{P}_{\leq \kappa_{2\alpha}(\dnfb)}(N)$. Assume that $p=\gtp(\ba/N)$ and $q=\gtp(\bb/N)$ for $ \ba, \bb \in \ce^{\alpha}$. 

Consider $\gtp(\ba \bb/N)$, then by local character there is $N_0 \lea N$ such that $\gtp(\ba \bb/N)$ does not fork over $N_0$ and $\| N_0 \| \leq \kappa_{2\alpha}(\dnfb)$. By symmetry and the hypothesis that $\dnfb \subseteq \nes$ we have that:

\[ N \nes_{N_0} \ba \bb. \]

Since $\gtp(\ba / N_0)= p\rest_{N_0}=q\rest_{N_0}=\gtp(\bb/ N_0)$ because $N_0$ is small, we have by the definition of explicitly non-splitting that $\gtp(\ba N/N_0)=\gtp(\bb N/ N_0)$. Hence $p = q$.
\end{proof}

The next result is the key result for many of the arguments given in this subsection. The idea of the proof is similar to that of the proof of the weak uniqueness property given in \cite[Theorem I.4.12]{van06}.

\begin{lemma}\label{sat-small}  Let $\mu, \kappa$ be infinite cardinals.
Assume $\dnfb$ is an almost-stable independence relation and $\mu \geq \kappa_{\kappa}(\dnfb) $. If $M$ is $\mu^+$-model homogeneous, $M \lea N$, $p, q \in \gS^{<\infty}(N)$, $p, q$ do not fork over $M$ and $p\rest_M=q\rest_M$, then $p^{I_0}\rest_A=q^{I_0}\rest_A$ for every $A \in \mathcal{P}_{< \mu^+} (N)$ and $I_0 \in \mathcal{P}_{<\kappa} ( |p|)$.
\end{lemma}
\begin{proof}
Let $A, I_0$ be as required and assume that  $p=\gtp(\ba/N)$, $q=\gtp(\bb/N)$ for $ \ba, \bb \in \ce^{\alpha}$ and $\alpha$ is an ordinal.

Consider $p^{I_0}\rest_M$ and $q^{I_0}\rest_M$ then by local character, base monotonicity and using that $|I_0| < \kappa$ there is $L \lea M$ such that $p^{I_0}\rest_M, q^{I_0}\rest_M$ do not fork over $L$ and $\| L \| \leq  \kappa_{\kappa}(\dnfb) \leq \mu $. 

Let $L'$ be the structure obtained by applying downward L\"{o}wenheim-Skolem to $L \cup A$ in $N$, observe that $\| L' \| \leq \mu$. Since $M$ is $\mu^+$-model homogeneous, there is $f: L'  \xrightarrow[L]{} M$.

Then by monotonicity, transitivity and the fact that $\dnfb \subseteq \nes$, we obtain that:

\[ \ba\rest_{I_0} \nes_{L} N \text{ and } \bb\rest_{I_0} \nes_{L} N. \]

Let $C_1= L'$ and $C_2 = f[L']$. Realize that $L \lea C_1, C_2 \lea N$ and $\gtp(C_1/L)=\gtp(C_2/L)$, then by the above equations, the definition of explicitly non-splitting and the choice of $C_1, C_2$ we obtain that:

\[ \gtp(\ba\rest_{I_0}L'/L)=\gtp(\ba\rest_{I_0}f[L']/L) \text{ and }  \gtp(\bb\rest_{I_0}L'/L)=\gtp(\bb\rest_{I_0}f[L']/L). \]

Since by hypothesis $p\rest_M=q\rest_M$ and $f[L']\lea M$, we have that $\gtp(\ba\rest_{I_0}/f[L'])=\gtp(\bb\rest_{I_0}/f[L'])$. Then it follows that $\gtp(\ba\rest_{I_0}f[L']/L)= \gtp(\bb\rest_{I_0}f[L']/L)$. Hence $\gtp(\ba\rest_{I_0}L'/L)=  \gtp(\bb\rest_{I_0}L'/L)$. Therefore, as $A \subseteq L'$, we conclude that $p^{I_0}\rest_A=q^{I_0}\rest_A$.
\end{proof}

\begin{remark}
For $\K$ an AEC with joint embedding, amalgamation and no maximal models, one can show, like in first-order, that if $\lambda \geq \kappa > \LS(\K)$, $M \in \K_{\leq \lambda}$  and $\lambda^{< \kappa}=\lambda$, then there is $N \in \K_\lambda$ such that $N$ is $\kappa$-Galois-saturated extending $M$. Moreover, $N$ is $\kappa$-model homogeneous as Shelah showed the equivalence between saturation and model homogeneity for AECs in \cite[\S II.1.14]{shelahaecbook}.
\end{remark}

We obtain a bound for almost-stable independence relations.

\begin{theorem}\label{star-bound}
If $\dnfb$ is an almost-stable independence relation, then
\[ NT(\mu, \lambda, \kappa) \leq \lambda^{(2^{\kappa_{2}(\dnfb)})} + \kappa^{-}. \]
\end{theorem}
\begin{proof}
Let $\lambda_0= \kappa_{2} (\dnfb)$, $\chi = \lambda^{2^{\lambda_0}} + \kappa^{-}$ and $\{ p_\alpha :\alpha < \chi^+ \} \subseteq \gS(M, \leq \mu)$ for $M \in \K_\lambda$. 

Observe that by the above remark there is $M'$ extending $M$ such that $M'$ is $(2^{\lambda_0})^+$-model homogeneous and $\| M' \|= \lambda^{2^{\lambda_0}}$.  For each $\alpha < \chi^+$, fix $q_\alpha \in \gS(M')$ such that $p_\alpha \leq q_\alpha$, this exist by amalgamation. Moreover, given $\alpha < \chi^+$, by local character there is $N \in K_{\lambda_0}$ such that $q_\alpha$ does not fork over $N$. Since $(2^{\lambda_0})^{\lambda_0} = 2^{\lambda_0}$, by the remark above there is $N'$ extending $N$ such that $N'$ is $(\lambda_0^+)$-model homogeneous and $\| N' \|= 2^{\lambda_0}$. Since $M'$ is $(2^{\lambda_0})^+$-model homogeneous, there is  $f: N' \xrightarrow[N]{} M'$. So fix $N_\alpha=f[N']$, realize $N_\alpha \in \K_{2^{\lambda_0}}$, $N_\alpha$ is $(\lambda_0^+)$-model homogeneous and $q_\alpha$ does not fork over $N_\alpha$ by base monotonicity. 

Define $\Phi: \chi^+ \to [M']^{2^{\lambda_0}}$ as $\Phi(\alpha)=N_\alpha$. Then by the pigeonhole principle  there is $N^* \in [M']^{2^{\lambda_0}}$ and $S \subseteq \chi^+$ of cardinality $\chi^+$ such that  $N_\alpha = N^*$ for every $\alpha \in S$. Now define $\Psi: S \to  \gS(N^*)$ as $\Psi(\alpha)= q_\alpha\rest_{N^*}$, since $|\gS(N^*)| \leq 2^{2^{\lambda_0}}$, by the pigeonhole principle  there is $q \in \gS(N^*)$ and $S'\subseteq S$ of size $\chi^+$ such that $q_\alpha\rest_{N^*}=q$ for every $\alpha \in S'$

 Observe that $q_\alpha \geq q \text{ and } q_\alpha \text{ does not fork over }  N^*$ for every $\alpha \in S'$. Then since $N^*$ is $(\lambda_0^+)$-model homogeneous and $\K$ is $\lambda_0$-tame (by Lemma \ref{star-tame}), it follows from Lemma \ref{sat-small} that $q_\alpha = q_\beta$ for every $\alpha, \beta \in S'$.  In particular, $\{p_\alpha : \alpha \in S'\}$ is consistent and $|S'|\geq \kappa$. Hence $NT(\mu, \lambda, \kappa) \leq  \lambda^{2^{\lambda_0}} + \kappa^-$. \end{proof}

The next results show that having an almost-stable independence relation implies that $\K$ is stable and that $\K$ does not have the tree property.

\begin{cor}\label{ntpp}
If $\dnfb$ is an almost-stable independence relation, then $\K$ is stable and $\K$ does not have the tree property.
\end{cor}
\begin{proof}
We show that $\K$ does not have the tree property by contradiction, the proof that $\K$ is stable is straightforward.
Let $\mu = (2^{\kappa_2(\dnfb)})^+$ and $\lambda=\beth_\mu(\mu)$. Since $\lambda^{<\mu}=\lambda$, the same construction as that of Lemma \ref{ntp} gives us that:
\[ \lambda^+ \leq NT(\mu, \lambda, (2^\mu)^+). \]

On the other hand, by the previous theorem we have that:

\[NT(\mu, \lambda, (2^\mu)^+) \leq \lambda^{2^{\kappa_2(\dnfb)} }+2^\mu= \lambda. \]

Putting together the last two equation we get a contradiction.  \end{proof}

The next result shows that an almost-stable independence relation is close to being a stable independence relation. Recall that $\K^{\mu^+\text{-mh}}$ is the $\mu^+$-AEC (see \cite{mu}) which models are the $\mu^+$-model homogeneous models of $\K$ and which order is the same as that of $\K$.

\begin{lemma}\label{star-sir}
Assume $\K$ is fully $(<\kappa)$-tame and -type-short.
If $\dnfb$ is an almost-stable independence relation and $\mu \geq \kappa_\kappa(\dnfb) + \kappa$, then $\K^{\mu^+\text{-mh}}$ has a stable independence relation. This is precisely the restriction of $\dnfb$ to $\mu^+$-model homogeneous models.
\end{lemma}
\begin{proof}
A big monster model of $\K$ is a monster model of $\K^{\mu^+\text{-mh}}$. For $M \in \K^{\mu^+\text{-mh}}$, $A,B \subseteq \ce$ define:
\[ A \overline{\dnf^{(*)}}_M B \text{ if and only if }  A \dnfb_M B.\]

We claim that $\overline{\dnf^{(*)}}$ is a stable independence relation in $\K^{\mu^+\text{-mh}}$. It is straightforward to show that it is an independence relation that satisfies symmetry. Uniqueness follows from Lemma \ref{sat-small}. As for local character, we have that given $\alpha$ and $p \in \gS^{\alpha}(M)$ with $M \in \K^{\mu^+\text{-mh}}$ there is $N\in  \K^{\mu^+\text{-mh}}$ such that $p$ does not $\overline{\dnf^{(*)}}$-forks over $N$ and $\| N \| \leq \kappa_\alpha(\dnfb) + \LS(\K)^{\mu}$. 
\end{proof}

We finish this section by showing that the results in this subsection can be used to obtain a new characterization of stability assuming simplicity for first-order theories. In order to present it, let us recall the notion of  non-splitting for first-order theories. A complete type $p$ in $\bar{x}$ does not split over $A$ a subset of the monster model if and only if for every $\bar{a}, \bar{b} \in Dom(p)$ and $\phi(\bar{x}, \bar{y})$ first-order formula, if  $tp(\bar{a}/A)=tp(\bar{b}/A)$, then $\phi(\bar{x}, \bar{a})\in p$ if and only if $\phi(\bar{x}, \bar{b})\in p$. This notion was introduced by Shelah in Definition 2.2 of \cite{sh3}.

\begin{lemma}\label{star-equiv}
Let $T$ be a simple complete first-order theory. The following are equivalent.
\begin{enumerate}
\item $\dnf_M \subseteq \dnf^{(ns)}_M$ for every $M$ model of $T$, where $\dnf$ denotes first-order non-forking and  $\dnf^{(ns)}$ denotes first-order non-splitting.
\item $T$ is stable.
\end{enumerate}
\end{lemma}
\begin{proof}
\fbox{$\to$}  Lemma \ref{sat-small}, Theorem \ref{star-bound} and Corollary \ref{ntpp} can be carried out if one replaces explicitly non-splitting for non-splitting in complete first-order theories.

\fbox{$\leftarrow$} Since $T$ is stable, non-forking has uniqueness (stationarity) over models. Under this hypothesis it is easy to show that $\dnf_M \subseteq \dnf^{(ns)}_M$ for every $M$ model of $T$ (a proof is given in \cite[4.2]{bgkv}). \end{proof}

\section{Simple independence relations}

We introduce simple independence relations and begin their study. We bound the possible values of $NT(\text{-},\text{-},\text{-})$ under the existence of a  simple independence relation and as a corollary we are able to show the failure of the $2$-tree property. As in the previous section we are assuming Hypothesis \ref{hyp}.

\begin{defin}\label{simple-ind}
 $\dnfb$ is a \emph{simple independence relation} in $\K$ if the following properties hold:
\begin{enumerate}
\item $\dnfb$ is an independence relation. 
\item (Symmetry)  $A \dnfb_{M}  B$ if and only if $B \dnfb_{M}  A$.
\item (Type-amalgamation) If $p \in \gS^{<\infty}(M)$, $M \subseteq A, B \subseteq  N$ and $A\dnfb_{M}  B$, then for all $q_1 \in \gS^{<\infty}(A; \ce), q_2 \in \gS^{<\infty}(B ; \ce)$ and $N^*\supseteq A, B$ such that $q_1,q_2 \geq p$ and $q_1, q_2$ do not fork over $M$, there exists $q \in \gS^{<\infty}(N^*)$ such that $q\geq q_1, q_2$ and $q$ does not fork over $M$. 
\item (Uniform local character) There exists $\theta$ and $\lambda$ cardinals such that: If $p \in \gS^{\alpha}(M)$, then there exists $M_0 \lea M$ with $|| M_0 || \leq \lambda + \alpha^{<\theta}$ and $p$ does not fork over $M_0$.  Recall that $(\kappa(\dnfb), \ell(\dnfb))$ are the least $(\lambda, \theta)$ with such a property. 
\end{enumerate}

\end{defin}

\begin{remark}
Let $T$ be a complete first-order theory. If $T$ is simple and $\dnfb$ is first-order non-forking, then $\dnfb$ is a simple independence relation. 
\end{remark}

\begin{remark}
The only difference between  stable independence relations and simple independence relations are conditions (3) and (4). As for (3), while we assume uniqueness in stable independence relations, we only assume type-amalgamation in simple independence relations. Although this may seem like a minor change, based on our knowledge of forking in first-order theories this is actually a significant one. As for (4), this is a minor change and we give natural conditions under which local character implies uniform local character (see Fact \ref{lc} and Corollary \ref{lc2}). 

\end{remark}

The next strengthening of the witness property is the key property to show that stable independence relations are simple independence relations if the AEC is tame and type-short.
\begin{defin}
Let $\dnfb$ be an independence relation.
$\dnfb$ has the \emph{$(<\theta)$-strong witness property} if for all $M \lea N$, $\alpha$ ordinals, and $\bb \in \ce^{\alpha}$: $\bb \dnfb_M  N$ if and only if $\bb\rest_{I} \dnfb_M A$ for every  $ A \in \mathcal{P}_{<\theta}( N)$ and $I \in \mathcal{P}_{<\theta}( \alpha)$.
\end{defin}

The proof of the following fact is the same as that of \cite[8.10]{lrv}, since the hypothesis are slightly different and the proof is short we repeat the argument for the convenience of the reader.

\begin{fact}\label{lc} Let $\dnfb$ be an independence relation.
If $\dnfb$ has local character and the $(< \theta)$-strong witness property, then $\dnfb$ has uniform local character.
\end{fact}
\begin{proof}
Since $\dnfb$ has local character, for each $\alpha < \theta$ we have that $\kappa_{\alpha}(\dnfb) < \infty$. Let $\lambda_0=sup\{ \kappa_{\alpha}(\dnfb) : \alpha < \theta \}$. We show that the pair $(\lambda_0, \theta)$ is a witness for uniform local character.

Let $M \in \K$ and $p=\gtp(\bb/M) \in \gS^\beta(M)$. For each $I \subseteq \beta$ with $|I|< \theta$, let $M_I \in [M]^{\lambda_0}$ such that $\bb\rest_{I} \dnfb_{M_I} M$, this exists by the choice of $\lambda_0$. Let $A = \bigcup_{I\subseteq \beta, |I| < \theta} M_I$ and $M_0$ be the structure obtained by applying downward L\"{o}wenheim-Skolem in $M$ to $A$. Observe that $\| M_0 \| \leq \lambda_0 + \beta^{<\theta}$
 and the $(< \theta)$-strong witness property together with monotonicity imply that $\bb \dnfb_{M_0} M$.
\end{proof}

The next lemma gives a condition under which a stable independence relation is a simple independence relation.

\begin{lemma}\label{st-simple} 
If $\dnfb$ is a stable independence relation that has the $(< \theta)$-strong witness property, then $\dnfb$ is a simple independence relation.
\end{lemma}
\begin{proof}
We only need to check properties (3) and (4). As for (4), this follows from Fact \ref{lc}. So we only need to show the type-amalgamation property.

Let $p \in \gS^{<\infty}(M)$, $M \subseteq A, B \subseteq \ce$, $A\dnfb_{M}  B$, $q_1 \in \gS^{<\infty}(A ; \ce)$ and $q_2 \in \gS^{<\infty}(B; \ce)$ and $N^* \supseteq A, B$ such that $q_1,q_2 \geq p$ and $q_1, q_2$ do not fork over $M$. Since $p \in \gS^{<\infty}(M)$ and $M \lea N^*$, by the extension property there is $q \in \gS^{<\infty}( N^*)$ such that $q \geq p$ and $q$ does not fork over $M$.

Observe that $q\rest_{A}, q_1 \in \gS^{<\infty}(A, \ce)$,   $q\rest_{A}, q_1 $ do not fork over $M$ and $(q\rest_{A})\rest_{M} = p = q_1\rest_{M}$, then by the uniqueness property ((3) of Definition \ref{sta-ind}) we have that $q\rest_{A}= q_1$. Hence $q_1 \leq q$. One can similarly show that $q\rest_{B} = q_2$. 

Therefore, $q \geq q_1, q_2$ and $q$ does not fork over $M$.  \end{proof}

The next assertion gives a natural assumption on $\K$ that implies the $(<\theta)$-strong witness property. The proof is similar to that of \cite[8.8]{lrv}, but we obtain a stronger result.

\begin{fact}\label{lc3} 
If $\K$ is fully $(<\theta)$-tame and -type-short and $\dnfb$ is a stable independence relation, then $\dnfb $ has the $(<\theta)$-strong witness property. 
\end{fact}
\begin{proof}
Let $M \lea N$ and $\bb \in \ce^{\alpha}$ such that $\bb\rest_{I} \dnfb_M A$ for every  $ A \in \mathcal{P}_{<\theta}( N)$ and $I \in \mathcal{P}_{<\theta}( \alpha)$. Let $p = \gtp(\bb/N)$ and $q \in \gS(N)$ such that $q$ does not fork over $M$ and $q$ extends $p\rest_M$, $q$ exists because of the extension property. Using that $\K$ is fully $(<\theta)$-tame and -type-short together with the uniqueness property one can show that $p = q$. As $q$ does not fork over $M$ by construction, it follows that $\bb \dnfb_M N$.
\end{proof}

\begin{cor}\label{lc2}
If  $\K$ is fully $(<\theta)$-tame and -type-short and $\dnfb$ is a stable independence relation, then $\dnfb$ is a simple independence relation.
\end{cor}


The next technical proposition is important as it shows that even when we are considering independence relations over sets in some sense models are ubiquitous

\begin{prop}\label{close}  Let $\dnfb$ be a simple independence relation.
If $A \dnfb_M B$, then there is $M^* \in \K$ with $B \cup M \subseteq M^*$ and $A \dnfb_M M^*$. 
\end{prop} 
\begin{proof} Assume $A \dnfb_M B$. By normality and monotonicity we can conclude that $A \dnfb_M B\cup M$. Let $M' \in \K$ the structure obtained by applying downward L\"{o}wenheim-Skolem in $\ce$ to $M \cup B \subseteq M'$.

Consider $p = \gtp(A/M)$, $q_1= \gtp(A/M \cup B)$ and $q_2 = \gtp(A/M)$. Observe that $p\leq q_1, q_2$, $q_1 \in \gS^{< \infty}(M \cup B; \ce)$ does not fork over $M$, $q_2 \in \gS^{< \infty}(M)$ does not fork over $M$, $M \subseteq M\cup B, M \subseteq M'$ and $M\cup B \dnfb_M M$. Recognize that $p,q_1, q_2$ and $M \subseteq M,M\cup B \subseteq M'$ satisfy the hypothesis of the type-amalgamation property, then there is $r \in \gS^{< \infty}(M') \geq q_1, q_2$ such that $r$  does not fork over $M$.

Suppose that $r=\gtp(A'/M')$, since $r \geq q_1$ there is $f \in Aut_{M\cup B}(\ce)$  such that $f[A']=A$. Since $r$ does not fork over $M$, we have that $A' \dnfb_M M'$. Then by invariance $f[A'] \dnfb_{f[M]} f[M']$. Observe $f[A']=A$, $f[M]= M$, so $A \dnfb_M f[M']$. Finally, realize that $M \cup B \subseteq f[M']$, hence $M^* :=f[M']$ satisfies what is needed. 
\end{proof}

The following notion generalizes the chain condition introduced in \cite[2.3]{les}.
\begin{defin}\label{dBound}
Let $\iota$ be an infinite cardinal. We say $\dnfb$ has the $\iota$-bound condition if: $\forall \lambda \in[\kappa(\dnfb), \infty) \forall M \in \K_\lambda \forall \kappa \in [\LS(\K),\lambda] \forall p \in \gS(M, \kappa) \forall \mu \in [\kappa(\dnfb) + \kappa, \lambda]$( If $\mu^{<\ell(\dnfb)} =\mu$ and $\{ p_\alpha : \alpha < (2^\mu)^+ \} \subseteq \gS(M, \leq \mu)$ are such that $p_\alpha \text{ is a non-forking extension of } p$ for every $\alpha < (2^\mu)^+$, then there are $A \subseteq   (2^\mu)^+$  and $q$ a type such that $|A|=\iota$ and $q$ is an extension of $p_\alpha$ for every $\alpha \in A$ ). Moreover, we say that $\dnfb$ has the strong $\iota$-bound condition if the type $q$ is a non-forking extension of $p$.
\end{defin}

The following is a generalization of \cite[2.4]{les}, which is based on an argument of Shelah which appeared in \cite[4.9]{gil}. Compared to \cite[2.4]{les}, instead of showing that two types are comparable we show that countably many types are comparable,  \cite[2.5]{les} mentions that this can be done in the first-order case. We have decided to present the argument to show that it does come through in this more general setting and because we will extend it in Lemma \ref{bound_cond}.

\begin{lemma}\label{chain}
If $\dnfb$ is a simple independence relation, then $\dnfb$ has the $\aleph_0$-bound condition.\footnote{Symmetry is not used to obtain this result.}
\end{lemma}
\begin{proof}
Let $\lambda, \mu, \kappa \in Car$, $M \in \K_\lambda$, $R \in [M]^{\kappa}$, $p \in \gS(R)$ and $\{p_\alpha \in \gS(N_\alpha) : \alpha < (2^\mu)^+\} \subseteq \gS(M, \leq \mu)$ be as in the definition of the $\aleph_0$-bound condition. By extension and transitivity, we may assume that all $N_\alpha$ have size $\mu$. 

We build $\{ M_\alpha: \alpha < (2^\mu)^+ \}$ strictly increasing and continuous chain such that:
\begin{enumerate}
\item $\forall \alpha \in  (2^\mu)^+( M_\alpha \in \K_{2^\mu})$.
\item $R \lea M_0$.
\item $\forall \alpha \in  (2^\mu)^+( N_\alpha \lea M_{\alpha +1} )$
\end{enumerate}

Let $S = \{ \alpha < (2^\mu)^+ : cf(\alpha) = \mu^+\}$ and $\Phi: S \to (2^\mu)^+$ be defined as $\Phi(\alpha)= min\{ \beta : \gtp(N_\alpha/ M_\alpha) \text{ does not fork over } M_\beta\}$. Observe that $\Phi$ is regressive by local character and the fact that $\mu^{<\ell(\dnfb)}=\mu$. Then by Fodor's lemma there is $S^* \subseteq S$ stationary and $\alpha^* < (2^\mu)^+$ such that $\forall \alpha \in S^*(\gtp(N_\alpha/M_\alpha) \text{ does not fork over } M_{\alpha^*})$. We may assume without loss of generality that $S= S^*$ and $\alpha^*=0$. Hence,

\begin{equation}
\forall \alpha \in S(\gtp(N_\alpha/M_\alpha) \text{ does not fork over } M_{0}).
\end{equation}

By local character and using again that $\mu^{< \ell(\dnfb)}= \mu$ we have that for all $\alpha \in S$ there is $R_\alpha \in [M_0]^{\mu}$ such that $\gtp(N_\alpha/M_\alpha)\rest_{M_0} \text{ does not fork over } R_\alpha$. Define $\Psi: S \to [M_0]^\mu$ as $\Psi(\alpha)= R_\alpha$. Then  by the pigeonhole principle, since $|[M_0]^\mu|=2^\mu$, we may assume that there is a $R^*\in [M_0]^\mu$ such that: 

\begin{equation}
\forall \alpha \in S (\gtp(N_\alpha/M_\alpha)\rest_{M_0} \text{ does not fork over } R^*).
\end{equation}

 By base monotonicity we may further assume that $R \lea R^*$. Then applying transitivity to the previous two equations  we obtain that:
\begin{equation}
\forall \alpha \in S( N_\alpha \dnfb_{R^*} M_\alpha).
\end{equation}

Moreover, given $\alpha \in S$ $p_\alpha \in \gS(N_\alpha)$ does not fork over $R$ and $N_\alpha \lea M_{\alpha + 1}$. Applying extension and transitivity, there is $q_\alpha \in \gS(M_{\alpha +1})$ extending $p_\alpha$ and $q_\alpha$ does not fork over $R$. By base monotonicity, since $R\lea R^*\lea M_{\alpha +1}$, we also have that $q_\alpha$ does not fork over $R^*$.

Let $\Upsilon: S \to \gS(R^*)$ be defined as $\Upsilon(\alpha)=q_\alpha\rest_{R^*}$, by the pigeonhole principle  we may assume that there is $q \in \gS(R^*)$ such that:
\begin{equation}
\forall \alpha \in S( q_\alpha \geq q \text{ and } q_\alpha \text{ does not fork over } R^*).
\end{equation}

Let $\{ \alpha_n : n \in \omega\} \subseteq S$ be an increasing set of ordinals. We build $\{r_n : n \in \omega \}$ such that:
\begin{enumerate}
\item $r_0=q_{\alpha_0}$.
\item $r_{n+1} \geq r_n, p_{\alpha_{n+1}}$.
\item $r_n \in \gS(M_{\alpha_{n} +1})$.
\item $r_n$ does not fork over $R$. 
\end{enumerate}

The base step is given so let us do the induction step. By equation (5)  $N_{\alpha_{n+1}} \dnfb_{R^*} M_{\alpha_{n+1}}$. Since $\alpha_n + 1 \leq \alpha_{n+1} \in S$, we have that $M_{\alpha_n +1} \lea M_{\alpha_{n+1}}$, so by monotonicity $N_{\alpha_{n+1}} \dnfb_{R^*} M_{\alpha_n +1}$ and by normality we have that  $N_{\alpha_{n+1}} \cup R^* \dnfb_{R^*} M_{\alpha_n + 1}$. Realize that $q \in \gS(R^*)$, $q_{\alpha_{n+1}}\rest_{N_{\alpha_{n+1}} \cup R^*} \in \gS(N_{\alpha_{n+1}} \cup R^*; \ce)$, $r_{n} \in \gS(M_{\alpha_n + 1})$ and $M_{\alpha_{n+1} +1}$ substituted by $p$, $q_1$, $q_2$ and $N^*$ satisfy the hypothesis of the type-amalgamation property. Therefore there is $r_{n+1}\in \gS(M_{\alpha_{n+1} +1})$ such that $r_{n+1} \geq q_{\alpha_{n+1}}\rest_{N_{\alpha_{n+1}} \cup R^*}, r_n$ and $r_{n+1}$ does not fork over $R^*$. 

In particular we have that $r_{n+1} \geq r_n, p_{\alpha_{n+1}}$  (since $q_{\alpha_{n+1}} \geq p_{\alpha_{n+1}}$) and by transitivity (since $r_{n+1} \geq r_n$, $R^*\leq M_{\alpha_n+1}$, and $r_n$ does not fork over $R$) we have that $r_{n+1}$ does not for over $R$. This finishes the construction.

Finally $\{r_n \in \gS(M_{{\alpha_n} +1}) : n \in \omega\}$ is an increasing chain of types so by \cite[11.3]{baldwinbook09}, there is $r^* \in \gS(\bigcup_{n\in \omega}M_{\alpha_n+1})$ such that $r^* \geq r_n$ for each $n \in \omega$. In particular, by clause (2) of the construction, we have that $r^*$ extends $p_{\alpha_n}$ for every $n < \omega$, which is precisely what we needed to show. \end{proof}

The following generalizes \cite[A]{les} to the AEC context. The proof is similar to that of Theorem \ref{tp-st}, but using the $\aleph_0$-bound condition instead of the uniqueness property. 

\begin{theorem}\label{bsimple}
If $\dnfb$ is a simple independence relation, $\kappa(\dnfb)  \leq \mu \leq \lambda$ and $\mu^{<\ell(\dnfb)}=\mu$, then \[ NT(\mu, \lambda, \aleph_0) \leq \lambda^{\kappa(\dnfb) } + 2^{\mu}. \]
In particular, $NT(\mu, \lambda) \leq \lambda^{\kappa(\dnfb)} + 2^{\mu}$
\end{theorem}
\begin{proof}
Let $\lambda_0= \kappa(\dnfb)$, $\chi = \lambda^{\lambda_0} + 2^{\mu}$ and $\{ p_\alpha \in \gS(N_\alpha): \alpha < \chi^+ \} \subseteq \gS(M, \leq \mu)$ where $M \in \K_\lambda$. Observe that by the extension property we may assume that each $N_\alpha \in \K_\mu$. As in the proof of Theorem \ref{tp-st} there are $S \subseteq \chi^+$ of size $\chi^+$, $R \in [M]^{\lambda_0}$ and $p \in\gS (R)$ such that for every $\alpha \in S$ $p_\alpha \geq p \text{ and } p_\alpha \text{ does not fork over } R$.

By the $\aleph_0$-bound condition, where the cardinal parameters are as in the definition except that $\kappa := \lambda_0$ and all the model theoretic parameters are the same with $\{ p_\alpha : \alpha \in S\}$ being the collection of types and $dom(p)=R$, we obtain that there are countable $A \subseteq S$ and $q$ a type such that $q \geq p_\alpha$ for each $\alpha \in A$. In particular $\{p_\alpha: \alpha \in A \}$ is consistent. Hence $NT(\mu,\lambda, \aleph_0) \leq \lambda^{\lambda_0} + 2^\mu$.
\end{proof}
 
\begin{remark}
Observe that when $\dnfb$ is a stable or almost-stable  independence relation Theorems \ref{tp-st} and \ref{star-bound} give us a better bound. Moreover,  Theorems \ref{tp-st} and \ref{star-bound} give us a bound for each $\kappa \in Car$ while the above corollary only gives us a bound when $\kappa$ is countable, as we will see in Theorem \ref{bound_cond2} more can be said if we assume the $(<\aleph_0)$-witness property. 
\end{remark}
The following result shows that we can not have the $2$-tree property if $\K$ has a simple independence relation.

\begin{cor}\label{n2-tp}
If $\dnfb$ is a simple independent relation, then $\K$ does not have the $2$-tree property.
\end{cor}
\begin{proof}
Suppose for the sake of contradiction that $\K$ has the $2$-tree property. 

Let $\lambda_0 =\kappa(\dnfb)$, $\mu = (\beth_{(\aleph_0 + \ell(\dnfb))^+}(\lambda_0^+))^+$ and $\lambda = \beth_{\mu}(\mu)$. Observe that the following cardinal arithmetic equalities hold:
\begin{enumerate}
\item $\mu^{<\ell(\dnfb)}=\mu$, using that $cf(\beth_{(\aleph_0 + \ell(\dnfb))^+}(\lambda_0^+))=(\aleph_0 + \ell(\dnfb))^+$ and Hausdorff formula.
\item $\lambda^{\lambda_0}+ 2^{\mu}= \lambda$, using that $cf(\lambda)=\mu > \lambda_0$ and that $\beth_\mu(\mu)>2^{\mu}$. 
\item $\lambda^{<\mu}=\lambda$, using that $cf(\lambda)=\mu$. 

\end{enumerate}

Applying Theorem \ref{bsimple}, this is possible by the first cardinal arithmetic equality, and by the second cardinal arithmetic equality we get that:

\begin{equation}
 NT(\mu, \lambda) \leq \lambda^{\lambda_0} + 2^\mu =\lambda. \end{equation}

Applying Lemma \ref{2-tp}, this is possible by the third cardinal arithmetic equality,  we get that \begin{equation}
\lambda^{\mu} \leq NT(\mu, \lambda).\end{equation}

So putting inequalities (7) and (8) we obtain that $\lambda^{\mu} \leq \lambda$, but this is a contradiction to K\"{o}nig's Lemma since $cf(\lambda)=\mu$.  \end{proof}

\begin{remark}
In the result above, instead of showing the failure of the $2$-tree property, we would have liked to obtain the failure of the tree property. We will show in Corollary \ref{bound_cond3} that this is the case if $\dnfb$ has the $(<\aleph_0)$-witness property for singletons.
\end{remark}

\section{Simple independent relations with the witness property}

In this section we continue the study of  simple independence relations under locality assumptions. We begin by showing the failure of the tree property under the existence of a simple independence relation with the $(< \aleph_0)$-witness property. Then we study  simple independence relations with the $(< \LS(\K)^+)$-witness property and obtain some basic results.

\subsection{Failure of the tree property} The next argument extends the one presented in Lemma \ref{chain}.

\begin{lemma}\label{bound_cond}
If $\dnfb$ is a simple independence relation with the $(<\aleph_0)$-witness property for singletons, then $\dnfb$ has the strong $(2^\mu)^+$-bound condition.
\end{lemma}
\begin{proof}[Proof sketch ]
Everything is the same as the proof of Lemma \ref{chain} until equation (6), but in this case instead of building only countably many $r_n's$ we will build $(2^\mu)^+$ many of them. 

Let $\{ \alpha_i : i < (2^\mu)^+\} \subseteq S$ be an increasing set of ordinals. We build $\{ r_i : i < (2^\mu)^+ \}$, $\{ a_i :  i < (2^\mu)^+\}$ and $\{ f_{j,i}:  j < i < (2^\mu)^+ \}$ such that:
\begin{enumerate}
\item $r_0 = q_{\alpha_0}=\gtp(a_0/ M_{\alpha_0 + 1})$.
\item If $k < j <i < (2^\mu)^+$, then $f_{k,i}=f_{j,i}\circ f_{k,j}$.
\item  $\forall j < i ( f_{j,i}\rest_{M_{\alpha_j +1}}= \id_{M_{\alpha_j +1}}, f_{j,i}(a_j)=a_i \text{ and } f_{j,i}\in Aut(\ce))$. 
\item $r_i=\gtp(a_i/M_{\alpha_i +1})$ does not fork over $R$.
\item $r_i \geq p_{\alpha_i}$. 
\item  $\forall j < i ( r_j \leq r_i)$.  
\end{enumerate} 

The construction in the successor step is similar to that of Lemma \ref{chain}, so we only show how to do the the step when $i$ is a limit ordinal. Since $\{r_j : j < i\}$, $\{ a_j : j < i \}$ and  $\{f_{k,j} :  k < j < i \}$ is a directed system, by \cite[11.3]{baldwinbook09}, there is $p^*=\gtp(a^*/ \bigcup _{j<i} M_{\alpha_j +1})$ upper bound for $\{r_j : j<i \}$ and $\{ f_{j,i}^* : j <i\}$ satisfying (2) and (3) but with $a^*$ substituted for $a_i$. 

Using the $(<\aleph_0)$-witness property, invariance and monotonicity it is easy to show that $p^*$ does not fork over $R$. Observe that $\bigcup _{j<i} M_{\alpha_j +1} \subseteq M_{\alpha_i}$, $N_{\alpha_i} \dnfb_{R^*} M_{\alpha_i}$ (by equation (5) of Lemma \ref{chain}) and $p^* \geq r_0$. Using these, one can show that $q \in \gS(R^*)$, $q_{\alpha_i}\rest_{N_{\alpha_i} \cup R^*}\in \gS(N_{\alpha_i} \cup R^*; \ce)$, $p^* \in \gS(\bigcup _{j<i} M_{\alpha_j +1})$ and $M_{\alpha_i + 1}$ substituted for $p$, $q_1$, $q_2$ and $N^*$ satisfy the hypothesis of the type-amalgamation property. Therefore, there is $r_{i}\in \gS(M_{\alpha_{i} +1})$ such that $r_{i} \geq q_{\alpha_{i}}\rest_{N_{\alpha_{i}} \cup R^*}, p^*$ and $r_{i}$ does not fork over $R^*$.

 Let $r_i:=\gtp(a_i/ M_{\alpha_{i} +1})$. Since $r_i\rest_{\bigcup _{j<i} M_{\alpha_j +1}}=p^*$, there is $g\in Aut(\ce)$ such that $g(a^*)=a_i$ and $g\rest_{\bigcup _{j<i} M_{\alpha_j +1}}=id_{\bigcup _{j<i} M_{\alpha_j +1}}$. For each $j < i$, let $f_{j,i}:=g\circ f_{j,i}^*$. It is easy to show that $r_i, a_i, \{ f_{j,i} : j<i \}$ satisfy (1) through (6), for conditions (4)-(6) see the explanation given in Lemma \ref{chain} . This finishes the construction.

We have constructed $\{ (r_i,a_i, \{f_{k,j} : k<j<i\}): i < (2^\mu)^+\}$  a coherent sequence of types, then by \cite[11.3]{baldwinbook09} there is $r^* \in \gS(\bigcup_{i < (2^\mu)^+} M_{\alpha_i +1})$ such that $r^*$ extends $r_i$ for every $i < (2^\mu)^+$. In particular, $p_{\alpha_i} \leq r^*$ for every $i < (2^\mu)^+$, since by condition (5) $p_{\alpha_i} \leq r_i$ for each $i<(2^{\mu})^+$. Moreover, using the $(<\aleph_0) $-witness property it follows that $r^*$ does not fork over $R$. 
\end{proof}

Using the above result instead of Lemma \ref{chain} we are able to extend Theorem \ref{bsimple} to uncountable cardinals. As the proof is similar to that of  Theorem \ref{bsimple} we omit it.

\begin{theorem}\label{bound_cond2}
If $\dnfb$ is a simple independence relation with the $(<\aleph_0)$-witness property for singletons, $\kappa(\dnfb)  \leq \mu \leq \lambda$ and $\mu^{<\ell(\dnfb)}=\mu$, then \[ NT(\mu, \lambda, (2^\mu)^+) \leq \lambda^{\kappa(\dnfb)} + 2^{\mu}. \]
\end{theorem}

As a corollary we obtain the failure of the tree property.

\begin{cor}\label{bound_cond3}
If $\dnfb$ is a simple independence relation with the $(<\aleph_0)$-witness property for singletons, then $\K$ does not have the tree property.
\end{cor}
\begin{proof}[Proof sketch ] Let $\lambda_0 = \kappa(\dnfb)$.
Let  $\mu$ and $\lambda$ be as in Theorem \ref{n2-tp}, i.e., $\mu = (\beth_{(\aleph_0 + \ell(\dnfb))^+}(\lambda_0^+))^+$ and $\lambda = \beth_{\mu}(\mu)$. Then doing a similar construction to that of Lemma \ref{ntp} we get that:
\begin{equation} \lambda^+ \leq NT(\mu, \lambda, (2^{\mu})^+). 
\end{equation}
But by Theorem \ref{bound_cond2} we have that  $NT(\mu, \lambda, (2^{\mu})^+)\leq \lambda^{\lambda_0} + 2^\mu$, then by choice of $\mu$ and $\lambda$ we have that $\lambda^{\lambda_0} + 2^\mu =\lambda$, so:
\begin{equation} NT(\mu, \lambda, (2^{\mu})^+)\leq \lambda.
\end{equation}
Observe that equations (9) and (10) give us a contradiction.

\end{proof}

\begin{remark}
A trivial example of a simple independence relation with the  $(<\aleph_0)$-witness property for singletons is first-order non-forking in $T$ where $T$ is a complete first-order simple theory.  This follows from the fact that non-forking has finite character. 
\end{remark}

\subsection{Simple independence relations with the $(< \LS(\K)^+)$-witness property} We continue the study of simple independence relations but with the additional hypothesis of the $( < \LS(\K)^+)$-witness property for singletons.  Recall that we have shown that 
if $\kappa_1(\dnfb)=\LS(\K)$. then $\dnfb$ has the $( < \LS(\K)^+)$-witness property for singletons (Lemma \ref{lcwitness}).

The following simple proposition will be used to study the Lascar rank in the next section. 

\begin{prop}\label{small} Let $\dnfb$ be a simple independence relation with the $(<\LS(\K)^+)$-witness property for singletons. If $M \lea N$, $p \in \gS(M)$, $q \in \gS(N)$ and $q$ is a forking extension of $p$, then there is $M^* \lea N$ with $\| M^* \| = \| M \|$, $M \lea M^*$ and $q\rest_{M^*}$ is a forking extension of $p$.
\end{prop}
\begin{proof}
Assume that $q=\gtp(b/N)$.
Suppose for the sake of contradiction that it is not the case, hence for every $M^* \lea N$ with $\| M^* \| = \| M \|$ and $M \lea M^*$ it holds that $q\rest_{M^*}$ does not fork over $M$. We will show, using the $(<\LS(\K)^+)$-witness property  for singletons, that $b \dnfb_{M} N$. 

Let $A \subseteq N$ and $| A | \leq \LS(\K)$, then apply downward L\"{o}wenheim-Skolem to $A \cup M$ inside $N$ to get $M^* \in \K_{\|M\|}$ such that $A \cup M \subseteq M^* \lea N$. Then by assumption $b \dnfb_{M} M^*$. So by monotonicity $b \dnfb_M A$. Therefore, by the $(<\LS(\K)^+)$-witness property  for singletons, we conclude that $b \dnfb_{M} N$, which contradicts the hypothesis that $q$ forks over $M$. 
\end{proof}

The next lemma generalizes \cite[2.3.7]{kim}.

\begin{lemma}

Let $\dnfb$ be a simple independence relation that has the $(< \LS(\K)^+)$-witness property for singletons and without uniform local character. The following are equivalent.
\begin{enumerate}
\item $\kappa_1(\dnfb) \leq \lambda$.
\item There are no $\{ M_i : i \leq \lambda^+ \}$ and $p \in \gS(M_{\lambda^+})$ such that $\{ M_i : i \leq \lambda^+ \}$ is strictly increasing and continuous chain and $p$ forks over $M_i$ for every $i < \lambda^+$.\footnote{This generalizes the first-order notion of a forking chain.}
\end{enumerate} 
\end{lemma}

\begin{proof}
\fbox{$\to$} Assume for the sake of contradiction that there is $\{ M_i : i \leq \lambda^+ \}$ a strictly increasing and continuous chain and  $p \in \gS(M_{\lambda^+})$ such that  $p$ forks over $M_i$ for every $i < \lambda^+$. Then by hypothesis there is $M' \in [M_{\lambda^+}]^{\lambda}$ such that $p$ does not fork over $M'$. Then by regularity of $\lambda^+$ and base monotonicity there is $i < \lambda^+$ such that $p_{\lambda^+}$ does not fork over $M_i$. This is a contradiction.

\fbox{$\leftarrow$} Assume for the sake of contradiction that $\kappa_1(\dnfb) > \lambda$, then there is $q = \gtp(a/N) \in \gS(N)$ such that $q$ forks over $M$ for every $M \in [N]^\lambda$. Realize that $\| N \| \geq \lambda^+$ as $q$ does not fork over $N$.

We build $\{ M_i : i < \lambda^+ \}$ strictly increasing and continuous chain such that:
\begin{enumerate}
\item For every $i < \lambda^+$, $M_i \in \K_\lambda$ and $M_i \lea N$. 
\item For every $j > i$, $q\rest_{M_j}$ forks over $M_i$. 
\end{enumerate}

Before we do the construction observe that this is enough by taking $M_{\lambda^+} = \bigcup_{i < \lambda^+} M_i$, $\{ M_i : i \leq \lambda^+ \}$ and $p=q\rest_{M_{\lambda^+}}$ .

In the base step, just take any $M_0 \in [N]^\lambda$. If $i < \lambda^+$ limit take unions and and it works by monotonicity, so the only interesting case is when $i =j+1$. Then by the $(< \LS(\K)^+)$-witness property there is $B \subseteq N$ of size $\LS(\K)$ such that $q\rest_B$ forks over $M_j$ and pick $c \in N \backslash M_j$. Let $M_{j+1}$ be the structure obtained by applying downward L\"{o}wenheim-Skolem to $B \cup M_j \cup \{ c\}$ in $N$. This works by the choice of $B$ and monotonicity.
\end{proof}

Realize that even simple assertions as the ones above become very hard to prove or perhaps even false if the independence relation does not have some locality assumptions.

\section{Supersimple independence relations and the $U$-rank}

In this section we introduce supersimple independence relations and show that they can be characterized by the Lascar rank under a locality assumption on the independence relation. We also show that the existence of a supersimple independence relation implies the $(<\aleph_0)$-witness property for singletons in classes with intersections.

Let us introduce the notion of a supersimple independence relation.

\begin{defin}
 $\dnfb$ is a \emph{supersimple independence relation} if the following properties hold:
\begin{enumerate}
\item  $\dnfb$ is a simple independence relation.
\item (Finite local character) For every $\delta$ limit ordinal, $\{ M_i : i \leq \delta \}$ increasing and continouos chain and $p \in \gS(M_\delta)$, there is $i < \delta$ such that $p$ does not fork over $M_i$.
\end{enumerate}
 \end{defin}

\begin{remark}
Let $T$ be a complete first-order theory. If $T$ is supersimple and $\dnfb$ is first-order non-forking, then $\dnfb$ is a supersimple independence relation. 
\end{remark}

The following is straightforward but will be useful.

\begin{lemma}\label{sup-local}
If $\dnfb$ is a supersimple independence relation, then $\kappa_1(\dnfb)= \LS(\K)$.
\end{lemma}
\begin{proof}[Proof sketch] The proof can be done by induction on the cardinality of the domain of the type. The base step is clear because types do not fork over their domain and for the induction step use that $\dnfb$ has finite local character.
\end{proof}

The above lemma together with Lemma \ref{lcwitness} can be used to obtain the next result.

\begin{cor}
If $\dnfb$ is a supersimple independence relation, then $\dnfb$ has the $(< \LS(\K)^+)$-witness property for singletons.
\end{cor}

The next lemma shows that supersimplicty and stability imply superstability.

\begin{lemma}
If $\dnfb$ is a stable and supersimple independence relation, then $\K$ is Galois-stable for every in a tail of cardinals\footnote{This is equivalent to any notion of superstability in the context of AECs if one assume that the AEC has a monster model and is tame by \cite{grva} and \cite{vaseyt}.}.
\end{lemma}
\begin{proof}
Since $\dnfb$ is a stable independence relation, by Corollary \ref{st-st} $\K$ is a Galois-stable AEC, so let $\lambda_0$ be the first stability cardinal. We show by induction on $\mu \geq \lambda_0$ that $\K$ is $\mu$-Galois-stable.

The base step is clear, so let us do the induction step. We proceed by contradiction, let $M \in\K_\mu$ and $\{p_i : i < \mu^+ \} \subseteq \gS(M)$ be an enumeration of different Galois-types. Let $\{ M_\alpha : \alpha < \mu\} \subseteq \K_{<\mu}$ be an increasing chain of submodels of $M$ such that $\bigcup_{\alpha < \mu} M_\alpha = M$. Then by supersimplicity for every $i < \mu^+$ there is $\alpha_i < \mu$ such that $p_i$ does not fork over $M_{\alpha_i}$. Then by the pigeonhole principle and using that $\dnfb$ has uniqueness,  one can show (as in Theorem \ref{tp-st}) that there are $i \neq j < \mu^+$ such that $p_i = p_j$- This is clearly a contradiction. Therefore, $\K$ is $\mu$-Galois-stable.
\end{proof}

It is worth noticing that Lemma \ref{bound_cond} can be carried out with the finite local character assumption instead of the $(<\aleph_0)$-witness property for singletons. The idea is that by applying finite local character and transitivity in limit stages one can show that the types constructed in the limit stages do not fork over $R$ (where $R$ is the one introduced in condition (4) of Lemma \ref{bound_cond}). 

\begin{cor}\label{supersim}
If $\dnfb$ is a supersimple independence relation, then 
\begin{itemize}
\item  if $\kappa(\dnfb)  \leq \mu \leq \lambda$ and $\mu^{<\ell(\dnfb)}=\mu$, then \[ NT(\mu, \lambda, (2^\mu)^+) \leq \lambda^{\kappa(\dnfb)} + 2^{\mu}. \]
\item $\K$ does not have the tree property.

\end{itemize}
\end{cor}

\subsection{Lascar rank}  The Lascar rank was extended to the AEC context by Boney and the first author in \cite{bogr}.

\begin{defin}[{\cite[7.2]{bogr}}]
We define $U$ with domain a type and range an ordinal or $\infty$ by, for any $p \in \gS(M)$
\begin{enumerate}
\item $U(p) \geq 0$.
\item $U(p) \geq \alpha$ for $\alpha$ limit ordinal if and only if $U(p) \geq \beta$ for each $\beta < \alpha$.
\item $U(p)\geq \beta + 1$ if and only if there are $M' \gea M$  and $p' \in \gS(M')$ with $\| M' \| = \| M \|$, $p'$ is a forking extension of $p$ and $U(p') \geq \beta$. 
\item $U(p) = \alpha$ if and only if $U(p)\geq \alpha$ and it is not the case that $U(p) \geq \alpha + 1$.
\item $U(p) = \infty$ if and only if $U(p) \geq \alpha$ for each $\alpha$ ordinal.
\end{enumerate}
\end{defin}

The next couple of results show that $U$ is a well-behaved rank. The proofs are similar to the ones presented in \cite[\S 7]{bogr}, but we fix a minor mistake of \cite[\S 7]{bogr}. The arguments of \cite[\S 7]{bogr} only work when the models under consideration are all of the same size, we are able to extend the arguments for models of different sizes by using the $(<\LS(\K)^+)$-witness property, specifically Proposition \ref{small}.

\begin{lemma}\label{basic-u}
 Let $\dnfb$ be a simple independence relation with the $(<\LS(\K)^+)$-witness property for singletons, then the $U$-rank satisfies:
\begin{enumerate}
\item (\cite[7.4]{bogr}) \underline{Invariance}: If $p \in \gS(M)$ and $f: M \cong M'$, then $U(p)=U(f(p))$.
\item \underline{Monotonicity}: If $M \lea N$, $p \in \gS(M)$, $q \in \gS(N)$ and $p \leq q$, then $U(q)\leq U(p)$.
\end{enumerate} 
\end{lemma}
\begin{proof} We provide a proof for (2) based on \cite[7.3]{bogr}. We prove by induction on $\alpha$ that: if $p \leq q$, then if $U(q) \geq \alpha$, then $U(p) \geq \alpha$.  The base step and limit step are trivial so assume that $\alpha =\beta+1$ and that $U(q) \geq \beta +1$. By definition there is $N' \gea N$ and $q' \in \gS(N')$ with $\| N' \| = \| N \|$, $q' \geq q$, $q'$ forks over $N$ and $U(q') \geq \beta$. Observe that by monotonicity $q'$ forks over $M$ and clearly $q' \geq p$. Then by Proposition  \ref{small} there is $M' \gea M$ with $\|M' \| = \| M \|$, $q'\rest_{M'} \geq p$ and $q'\rest_{M'}$ forks over $M$. Since $q'\rest_{M'} \leq q'$, by induction hypothesis $U(q'\rest_{M'}) \geq \beta$. Therefore, by the definition of the $U$-rank $U(p) \geq \beta +1 $. \end{proof}

\begin{lemma}\label{u-fork}  Let $\dnfb$ be a simple independence relation with $(<\LS(\K)^+)$-witness property for singletons.
Let  $M \lea N$, $p \in \gS(M)$ and $q \in \gS(N)$ with $p \leq q$ and  $U(p), U(q) < \infty$. Then:
\[ U(p)=U(q) \text{ if and only if } q \text{ is a non-forking extension of } p. \]
\end{lemma}
\begin{proof}
\fbox{$\to$} Assume for a sake of contradiction that $q$ forks over $p$. Then by Proposition \ref{small} there is $M^* \in \K$ with $\| M^* \| = \| M \|$, $q\rest_{M^*}\geq p$ and $q\rest_{M^*}$ forks over $M$. Then from monotonicity of the rank and the definition of the $U$-rank, we can conclude that $U(p) \geq U(q) +1$, which clearly contradicts our hypothesis.

\fbox{$\leftarrow$} The same argument given in \cite[7.7]{bogr} can be carried out in our context due to Proposition \ref{close}. \end{proof}

\begin{fact}\label{ordbound}(\cite[7.8]{bogr})   Let $\dnfb$ be a simple independence relation with the $(<\LS(\K)^+)$-witness property for singletons. For each $\mu \geq \LS(\K)$, there is some $\alpha_{\K, \mu} < (2^{\mu})^+$ such that for any $M \in \K_\mu$, if $U(p) \geq \alpha_{\K, \mu}$, then $U(p)=\infty$.  
\end{fact}

The proof of the following lemma is similar to that of \cite[7.9]{bogr}.

\begin{lemma}\label{super}  Let $\dnfb$ be a simple independence relation with the $(<\LS(\K)^+)$-witness property for singletons. Let $M \in \K_\mu$ and $p \in \gS(M)$. The following are equivalent.
\begin{enumerate}
\item $U(p)=\infty$
\item There is an increasing chain of types $\{p_n : n < \omega \}$ such that $p_0=p$ and $p_{n+1} $ is a forking extension of $p_n$ for each $n< \omega$.
\end{enumerate}
\end{lemma}
\begin{proof}
\fbox{$\to$} Let $\alpha_{\K, \mu}$ be the ordinal given by Fact \ref{ordbound}.
We build $\{M_n : n < \omega\}$ and $\{ p_n \in \gS(M_n) : n <\omega\}$ by induction such that:
\begin{enumerate}
\item $p_0 = p$.
\item $M_n \in \K_\mu$.
\item $p_{n+1} $ is a forking extension of $p_n$ for every $n< \omega$.
\item $U(p_n) \geq \alpha_{K, \mu} + 1$.
\end{enumerate}
The base step is given by condition (1). As for the induction step, we have by induction that  $U(p_n) \geq \alpha_{K, \mu} + 1$. Then by definition of the $U$-rank there is $M_{n+1} \geq M_n$ and $p_{n+1} \in \gS(M_{n+1})$ a forking extension of $p_n$  such that $\| M_{n+1} \| = \| M_n \| = \mu$ and $U(p_{n+1}) \geq \alpha_{\K, \mu}$. Observe that since  $U(p_{n+1}) \geq \alpha_{\K, \mu}$ and $M_{n+1} \in \K_\mu$, we have that  $U(p_{n+1}) = \infty$, so  $U(p_{n+1}) \geq \alpha_{\K, \mu} +1$.

\fbox{$\leftarrow$}  Let $\{p_n : n < \omega \}$ be an increasing chain of types such that $p_0=p$ and $p_{n+1} $ is a forking extension of $p_n$ for each $n< \omega$. We prove by induction on $\alpha$ that:  $ U(p_n) \geq \alpha$ for every $n < \omega$. The base step and limit case are trivial so assume that $\alpha = \beta +1$ and take $n\in \omega$. By induction hypothesis $U(p_{n+1})\geq \beta$ and by hypothesis $p_{n+1} $ is a forking extension of $p_n$. Then by Proposition \ref{small} there is $M^* \in \K$ with $\| M^* \| = \| dom(p_n) \|$, $p_{n+1} \rest_{M^*}\geq p_n$   and $p_{n+1} \rest_{M^*}$ forks over $dom(p_n)$. Then by monotonicity of the rank and the definition of the $U$-rank we can conclude that $U(p_n) \geq \beta +1=\alpha$.
\end{proof}

With this we obtain our main result regarding the relationship between a supersimple independence relations and the $U$-rank . This generalizes a characterization of supersimplicity for first-order theories \cite[2.5.16]{kim}.

\begin{theorem}\label{equivsup}
Let $\dnfb$ be a simple independence relation with the $(<\aleph_0)$-witness property for singletons. The following are equivalent.
\begin{enumerate}
\item $\dnfb$ is a supersimple independence relation.
\item If $M \in \K$ and $p\in \gS(M)$, then $U(p) < \infty$.
\end{enumerate}
\end{theorem}
\begin{proof}
\fbox{$\to$} Suppose there are $M \in \K$ and $p \in \gS(M)$ such that $U(p)=\infty$. Then, by Lemma \ref{super}, there is an increasing chain of types $\{p_n  : n < \omega \}$ such that $p_0=p$ and $p_{n+1} $ is a forking extension of $p_n$ for every $n< \omega$.

Since we have that $\{p_n : n < \omega\}$ is an increasing chain of types, by \cite[11.3]{baldwinbook09}, there is $p_\omega \in \gS(\bigcup_{n< \omega} dom(p_n) )$ such that $p_\omega \geq p_n$ for each $n < \omega$. Then, by the definition of supersimplicty, there is $n < \omega$ such that $p_\omega$ does not fork over $dom(p_n)$. Hence by monotonicity $p_\omega\rest_{dom( p_{n+1} )}=p_{n+1}$ does not fork over $dom(p_n)$, which contradicts the fact that $p_{n+1} $ is a forking extension of $p_n$.

\fbox{$\leftarrow$} Assume for the sake of contradiction that $\dnfb$ is not a supersimple independence relation, then there are $\delta$ a limit ordinal, $\{N_i : i \leq \delta \}$ an increasing and continuous chain and $p \in \gS(N_\delta)$, such that $p$ forks over $N_i$ for every $i < \delta$. 

We first show that for every $i < \delta$ there is $j_i \in (i, \delta)$ such that $p\rest_{N_{j_i}}$ forks over $N_i$. 
Let $i < \delta$ and suppose for the sake of contradiction that $p\rest_{N_j}$ does not fork over $N_i$ for each $j \in (i, \delta)$. Then using the $(< \aleph_0)$-witness property for singletons, as in Proposition \ref{small}, one can show that $p$ does not fork over $N_i$, contradicting the hypothesis that $p$ forks over $N_i$.

Then one can build by induction $\{i_n : n< \omega\} \subseteq \delta$ increasing  such that $\{ p_{i_n} : n  < \omega\}$ is an increasing chain of types with $p_{i_{n+1}}$ a forking extension of $p_{i_n}$ for each $n < \omega$ where $p_{i_n}=p\rest_{N_{i_n}}$. Therefore, by Lemma \ref{super}, we can conclude that $U(p_{i_0})=\infty$. This contradicts the fact that  $U(p_{i_0}) < \infty$ by hypothesis.
\end{proof}

\subsection{A familiy of classes with the $(<\aleph_0)$-witness property} In this subsection we show that in classes that admit intersections one obtains the $(<\aleph_0)$-witness property for singletons from supersimplicity. Similar results assuming the existence of a superstable-like independence relation are obtained in Appendix C of \cite{vaseyu}. We begin by recalling the definition of classes that admit intersections, these were introduced by Baldwin and Shelah.

\begin{defin}  [{\cite[1.2]{BSh} }]An AEC \emph{admits intersections} if for every $N \in \K$ and $A \subseteq |N|$ there is $M_0 \lea N$ such that $|M_0|= \bigcap\{M \lea N : A \subseteq |M|\}$. For $N \in \K$ and $A \subseteq |N|$, let $cl^{N}_{\K}(A)=\bigcap\{M \lea N : A \subseteq |M|\}$, if it is clear from the context we will drop the $\K$. We write $cl(A)$ instead of $cl^{\ce}_{\K}(A)$ if $\ce$ is a monster model of $\K$ and $\K$ is clear from the context.
\end{defin}

Below we provide the properties of AECs that admit intersections that we will use, for a more detailed introduction to AECs that admit intersections the reader can consult \cite[$\S$2]{vaseyu}.

\begin{fact}\label{easy-f}
 Let $\K$ be an AEC that admits intersections.
\begin{enumerate}
\item If $A \subseteq B \subseteq N$, then $cl^N(A) \lea cl^N(B)$. 
\item If $A \subseteq M$ and $M \in \K$, then $cl(A) \lea M$.
\item (Finite character) Let $M \in \K$ and $a \in cl^M(B)$, then there is $B_0 \subseteq_{fin} B$ such that $a \in cl^M(B_0)$.
\end{enumerate}
\end{fact}
\begin{proof}
(1) and (2) are trivial and (3) is \cite[2.14]{vaseyu}.
\end{proof}

We show that finite local character is actually witnessed by a finite set in classes with intersections.

\begin{lemma} Let $\K$ be an AEC with a monster model that admits intersections and $\dnfb$ be a simple independence relation. The following are equivalent.
\begin{enumerate}
\item (Finite local character) For every $\delta$ limit ordinal, $\{ M_i : i \leq \delta \}$ increasing and continouos chain and $p \in \gS(M_\delta)$, there is $i < \delta$ such that $p$ does not fork over $M_i$.
\item For every $M \in \K$ and $p \in \gS(M)$, there is $D \subseteq_{\text{fin}} M$ such that $p$ does not fork over $cl(D)$.
\end{enumerate}
\end{lemma}
\begin{proof}
The backward direction follows trivially using monotonicity, so we show the forward direction.

Let $M \in \K$ and $p \in \gS(M)$, we show by induction on $\lambda \leq \| M \|$ the following:

\[ (*)_\lambda : \text{ For every } A \in \mathcal{P}_{\lambda}(M) \text{ and } p \in \gS(cl(A)) \text{ , there is } D \subseteq_{\text{fin}} M \text{ s.t. } p \text{ does not fork over } cl(D).\]

Observe that this is enough as $cl(M)=M$. So let us do the proof.

\underline{Base:} If $\lambda$ is finite $(*)_\lambda$ is clear because given $p \in \gS(cl(A))$, $p$ does not fork over $cl(A)$. So let us do the case when $\lambda = \aleph_0$. Let $A =\{a_i : i < \omega \}$ be an enumeration without repetitions and $p \in \gS(cl(A))$.  Let $M_i = cl( \{a_j : j < i\})$ for every $i < \omega$ and $M_\omega = \bigcup_{i < \omega} M_i$. Observe that $\{ M_i : i \leq \omega \}$ is an increasing and continuous chain and $\bigcup_{i < \omega} M_i = cl(A)$ by the finite character of the closure operator. Then by (1) there is $i < \omega$ such that $p$ does not fork over $M_i = cl( \{a_j : j < i\})$. So $D = \{a_j : j < i\}$ is as needed.

\underline{Induction step:} Let $\lambda$ be an uncountable cardinal and suppose that $(*)_\mu$ holds for every $\mu < \lambda$. In this case the proof is similar to that of the  base step when $\lambda = \aleph_0$. The only difference is that on top of using (1), one uses the induction hypothesis, and transitivity of the independence relation.  \end{proof}

\begin{cor}\label{lc4}
 Let $\K$ be a class that admits intersections. If $\dnfb$ is a supersimple independence relation, then $\dnfb$ has the $(< \aleph_0)$-witness property for singletons.
\end{cor}
\begin{proof}
Let $M \lea N$ and $a \in \ce$ such that $a \dnfb_M B$ for every $B \subseteq_{\text{fin}} N$.

By the previous lemma there is $D  \subseteq_{\text{fin}} N$ such that $a \dnfb_{cl(D)} N$, then by base monotonicity $a \dnfb_{cl(DM)} N$. On the other hand, by hypothesis $a \dnfb_M D$, then by normality, monotonicity and Proposition \ref{close} it follows that $a \dnfb_M  cl(DM)$. Therefore,  applying transitivity to $a \dnfb_M  cl(DM)$ and  $a \dnfb_{cl(DM)} N$ we obtain that $a \dnfb_M N$.
\end{proof}
 
\section{Future work}

In \cite[4.2]{kipi} it is shown that if a complete first-order theory is simple, then there is a canonical independence relation satisfying the type-amalgamation property. In \cite{bgkv} it is shown that stable independence relations are canonical.  So it is natural to ask if the same holds true for simple and supersimple independence relations.

\begin{question}\label{q-cano}
If $\K$ has $\dnfb$ a simple or supersimple independence relation, is $\dnfb$ canonical?
\end{question}

\begin{remark}\label{kams} Theorem 1.1 of \cite{kam} gives a positive answer to the above question under the assumptions that $\dnfb$ has the $(<\aleph_0)$-witness property.

\end{remark}

It is known that for a complete first-order theory $T$, $T$ is simple if and only if $T$ does not have the tree property (see for example \cite[3.10]{gil}). In Sections 5 and 6 we showed some instances of the forward direction for simple independence relations (Corollary \ref{n2-tp} and Corollary \ref{bound_cond3}). So we ask the following:

\begin{question}
 If $\K$ does not have the tree property, does $\K$ have $\dnfb$ a simple independence relation?
\end{question}

Another notion that we studied in this paper is that of the witness property for independence relations. This seems to be a very strong hypothesis that can be taken for granted in first-order theories as forking has finite character. Regarding it we ask:
\begin{question}
Can Fact \ref{lc3} be extended to simple independence relations? More precisely, if $\K$ is fully $(<\theta)$-tame and -type-short and $\dnfb$ is a simple independence relation, does $\dnfb$ have the $(<\theta)$-strong witness property?
\end{question}

A related question is the following:

\begin{question}
Is Corollary \ref{lc4} true for all AECs with a monster model?  
\end{question}

Moreover, we used the witness properties a few times in this paper, see for example Lemma \ref{bound_cond} and Theorem \ref{equivsup}. An interesting question would be if the use of the witness property is necessary in those arguments where we use it.

In \cite[8.16]{lrv} it is shown that the existence of a stable independence relation implies that the AEC is tame. We extended this result for almost-stable independence relations in Lemma \ref{star-tame}, so a natural question to ask is:

\begin{question}
If $\K$ has $\dnfb$ a simple or supersimple independence relation, is $\K$ tame?
\end{question}

Finally, as it was mentioned in the introduction, we think that it is premature to focus on applications. Nevertheless, we acknowledge the importance of \emph{good examples}. Below is a list of the type of examples that we are interested in. 

\begin{question}\
\begin{itemize}
\item Find an example of a simple independence relation that is not a stable independence relation in an AEC that is not fully $(<\aleph_0)$-tame and type-short.
\item Find an example of a supersimple independence relation that is not a stable independence relation in an AEC that is not fully $(<\aleph_0)$-tame and type-short.
\item Find an example of a strictly simple independence relation without the $(<\aleph_0)$-witness property.
\item Find an example of a strictly simple independence relation without the witness property. 
\item Find an example of a strictly simple independence relation without the witness property in an AEC that is fully $(<\aleph_0)$-tame and type-short.
\end{itemize}
\end{question}


\end{document}